\documentclass[lettersize]{IEEEtran}
\usepackage{amsmath,amsfonts}
\usepackage{amsthm}
\usepackage{algorithmic}
\usepackage[ruled]{algorithm}
\usepackage{array}
\usepackage[caption=false,font=normalsize,labelfont=sf,textfont=sf]{subfig}
\usepackage{textcomp}
\usepackage{stfloats}
\usepackage{url}
\usepackage{verbatim}
\usepackage{graphicx}
\usepackage{cite}
\usepackage[colorlinks=true, citecolor=blue, linkcolor=black, urlcolor=blue]{hyperref}
\usepackage{cleveref}

\usepackage{booktabs} 
\usepackage{multirow} 
\usepackage{makecell} 
\usepackage{csvsimple} 

\usepackage{enumitem}

\newtheorem{remark}{Remark}
\newtheorem{proposition}{Proposition}
\newtheorem{theorem}{Theorem}
\newtheorem{lemma}{Lemma}

\usepackage{amsopn}

\DeclareMathOperator{\dist}{dist}
\DeclareMathOperator{\dom}{dom}
\DeclareMathOperator*{\argmin}{arg\,min}
\DeclareMathOperator{\prox}{prox}
\DeclareMathOperator{\rank}{rank}
\DeclareMathOperator{\crit}{crit}

\newcommand{\Jcal}[1]{\mathcal{J}_{#1}}  

\usepackage{mathtools}

\DeclarePairedDelimiter{\basenorm}{\lVert}{\rVert}

\NewDocumentCommand{\norm}{ s o m }{%
  \IfBooleanTF{#1}
    {\basenorm*{#3}\IfValueT{#2}{_{#2}}}%
    {\basenorm{#3}\IfValueT{#2}{_{#2}}}%
}

\crefformat{equation}{\textup{#2(#1)#3}}
\crefrangeformat{equation}{\textup{#3(#1)#4--#5(#2)#6}}
\crefmultiformat{equation}{\textup{#2(#1)#3}}{ and \textup{#2(#1)#3}}
{, \textup{#2(#1)#3}}{, and \textup{#2(#1)#3}}
\crefrangemultiformat{equation}{\textup{#3(#1)#4--#5(#2)#6}}%
{ and \textup{#3(#1)#4--#5(#2)#6}}{, \textup{#3(#1)#4--#5(#2)#6}}{, and \textup{#3(#1)#4--#5(#2)#6}}

\Crefformat{equation}{#2Equation~\textup{(#1)}#3}
\Crefrangeformat{equation}{Equations~\textup{#3(#1)#4--#5(#2)#6}}
\Crefmultiformat{equation}{Equations~\textup{#2(#1)#3}}{ and \textup{#2(#1)#3}}
{, \textup{#2(#1)#3}}{, and \textup{#2(#1)#3}}
\Crefrangemultiformat{equation}{Equations~\textup{#3(#1)#4--#5(#2)#6}}%
{ and \textup{#3(#1)#4--#5(#2)#6}}{, \textup{#3(#1)#4--#5(#2)#6}}{, and \textup{#3(#1)#4--#5(#2)#6}}

\crefdefaultlabelformat{#2\textup{#1}#3}

\begin{document}

\title{Automatic Model-Order Selection for Nonnegative Matrix Factorization via Column $\ell_{2,0}$ Regularization}

\author{Yulin He, and Ran Gu%
\thanks{This work was funded by the National Key R\&D Program of China (2022YFA1003800), the National Natural Science Foundation of China (12201318), the Natural Science Foundation of Tianjin (25JCJQJC00300), and the Tianjin Science and Technology Program (24ZXZSSS00320).}
\thanks{Yulin He and Ran Gu are with NITFID, School of Statistics and Data Sciences, LPMC, KLMDASR, LEBPS and AAIS, Nankai University, Tianjin 300071, China (e-mail: 1120240083@mail.nankai.edu.cn; rgu@nankai.edu.cn).}}

\markboth{}{}

\maketitle

\begin{abstract}
Nonnegative matrix factorization represents nonnegative signals as additive combinations of latent components, but its factorization rank, and hence the model order, must usually be specified beforehand. An underestimated order discards signal structure, whereas an overestimated order produces redundant components and unstable decompositions. We propose a column $\ell_{2,0}$-regularized formulation that estimates the model order from an initial upper bound by suppressing inactive columns in both factors. A warm-started regularization path progressively removes redundant components without changing the factor dimensions, and a marginal reconstruction-loss criterion selects an order along the path. To solve the resulting nonconvex and discontinuous problem, we develop an inertial proximal alternating linearized minimization method, a scale-balanced variant, and a proximal active-set method based on P-stationarity. The balancing operation equalizes the norms of paired factor columns while preserving their rank-one products. We characterize the critical points and local minimizers of the model, provide a sufficient-condition result for rank recovery, and prove whole-sequence convergence of the proposed algorithms under explicit step-size and inertial-parameter conditions using the Kurdyka--\L{}ojasiewicz framework. Dedicated experiments show that the warm-started $\lambda$-path is more efficient than increasing- and decreasing-order discrete $r$-paths, while scale balancing yields a more stable rank-selection path. Experiments on synthetic data and diverse signal benchmarks show that both iPALM and PASM provide reliable model-order estimates with favorable computational efficiency.
\end{abstract}

\begin{IEEEkeywords}
column sparsity, model-order selection, nonconvex optimization, nonnegative matrix factorization, proximal alternating minimization
\end{IEEEkeywords}

\section{Introduction}

\IEEEPARstart{N}{onnegative} matrix factorization (NMF) represents a nonnegative data matrix $X\in\mathbb{R}_+^{n\times m}$ as
$X\approx UV^\top$, where $U\in\mathbb{R}_+^{n\times r}$ and $V\in\mathbb{R}_+^{m\times r}$.
The decomposition is additive: each term $U_{:,k}V_{:,k}^{\top}$ describes one latent component, while the two factor columns encode its signature and its coefficients across the observations.
This interpretation makes NMF particularly suitable for nonnegative signals whose components have physical meaning, such as spectral signatures and abundances in hyperspectral images or spectral templates and temporal activations in audio spectrograms \cite{lee1999learning,gillis2020nonnegative}.
Recent signal-processing variants incorporate convolutional structure for temporal signals \cite{degleris2020provably,malagnoux2026connecting}, source identifiability and perceptually motivated divergences for audio separation \cite{leplat2020blind}, continuous functional structure for sampled spectra \cite{hautecoeur2023least}, polarization constraints for multidimensional imaging \cite{flamant2020quaternion}, and measurement-specific noise models \cite{green2024algorithms}.

In these applications, the factorization rank $r$ is not merely a numerical parameter: it is interpreted as a model order, such as the number of active sources, endmembers, or signal components.
This order is rarely known in advance.
Choosing it too small merges or removes genuine components and increases approximation bias, whereas choosing it too large introduces redundant factors, increases sensitivity to noise and initialization, and complicates physical interpretation.
Throughout this paper, the target order refers to the number of underlying nonnegative components: it coincides with the nonnegative rank in the ideal noiseless model and is represented in measured data by the reported number of physical components (e.g., endmembers), allowing for small discrepancies due to noise and model mismatch.

Existing selection methods address this problem from several directions.
Residual-elbow, stability, information-criterion, and cross-validation procedures compare factorizations over a discrete grid of candidate ranks \cite{brunet2004metagenes,squires2017rank,muzzarelli2019rank,eswar2024rank}.
Although broadly applicable, they require repeated nonconvex optimizations and can be sensitive to the candidate grid, initialization, and decision threshold.
Singular-value hard thresholding is inexpensive, but it estimates an algebraic rank under a particular spectral-noise model rather than directly counting nonnegative components.
Bayesian nonparametric NMF can infer a latent model order jointly with the factors \cite{li2020bayesian}, at the cost of specifying probabilistic models and performing posterior inference.
In signal-processing tasks such as hyperspectral unmixing and blind source separation, model-order misspecification is particularly consequential because the selected order directly determines the number of estimated endmembers or sources and thus affects their estimation accuracy.
Existing order-estimation procedures often rely on statistical assumptions about the observed spectra or on discrete searches over candidate orders, which can be computationally expensive and insufficiently robust; the proposed framework instead enables end-to-end automatic order selection directly within these signal factorization models.
Structured NMF models offer another route: minimum-volume regularization may deactivate superfluous audio sources when the assumed order is too large \cite{leplat2020blind}, sparse NMF promotes parsimonious activations \cite{marmin2023majorization}, and overcomplete factor-and-merge strategies can improve the quality of a factorization when a target rank is supplied \cite{guo2025optimal}.
These approaches provide important application-specific mechanisms; a complementary objective is to count and eliminate paired rank-one components directly, starting only from a rough upper bound on the order.

Column sparsity provides a natural optimization mechanism for this purpose.
The column $\ell_{2,0}$ quantity counts nonzero columns and therefore acts directly on the number of active components, unlike entrywise sparsity penalties.
It has been used as a structural constraint on one NMF factor \cite{min2022structured} and as a regularizer for rank-adaptive low-rank matrix recovery \cite{tao2022column}.
Extending this idea to model-order selection in NMF is nevertheless nontrivial: the two nonnegative factors are coupled, their columns have reciprocal scaling ambiguities, the objective is discontinuous, and the selected factor support must be related carefully to the nonnegative rank of the reconstructed signal.

To address these issues, we consider the following column $\ell_{2,0}$-regularized NMF model:
{\small
\begin{multline}\label{eq:l20_regularized_nmf}
\min_{\substack{U \in \mathbb{R}_+^{n \times r},\\ V \in \mathbb{R}_+^{m \times r}}} \Phi_{\lambda,\mu}(U,V) := \frac{1}{2}\|X-UV^\top\|_F^2 + \frac{\mu}{2}\left( \|U\|_F^2 + \|V\|_F^2 \right) \\+ \lambda\left( \|U\|_{2,0} + \|V\|_{2,0} \right).
\end{multline}
}%
Here $r$ is an initial upper bound, $\|U\|_{2,0}$ and $\|V\|_{2,0}$ count active factor columns, and $\lambda>0$ controls component removal.
The symmetric quadratic term with a small $\mu>0$ controls the scale of the factors, yields bounded level sets, and selects balanced representatives from reciprocally scaled factor pairs.

We trace solutions over an increasing sequence of $\lambda$ values.
Each subproblem is initialized by the preceding solution, so redundant components are progressively removed while the factor dimensions remain fixed.
This warm-started continuation avoids repeatedly discarding learned columns or appending randomly initialized columns, as required by a discrete rank path.
The candidate orders along the path are compared through their marginal reconstruction losses, producing a fully data-driven order estimate.
The dedicated continuous-$\lambda$-path versus discrete-$r$-path experiment in \Cref{sec:lambda_path_analysis} verifies these advantages: \Cref{fig:lambda_path} illustrates solution reuse along rank plateaus, and \Cref{fig:time_comparison} compares the resulting total search times on the synthetic and Swimmer data.

The resulting problem is nonconvex and discontinuous, but its block proximal maps reduce to nonnegative projection followed by column hard thresholding.
We exploit this structure in an inertial proximal alternating linearized minimization method and in a proximal active-set method based on P-stationarity.
We also introduce a scale-balanced variant that equalizes the norms of paired factor columns while preserving every rank-one product and hence the reconstructed signal.
This operation mitigates numerical imbalance in the thresholding steps and is included explicitly in the convergence analysis.
The scale-balancing ablation in \Cref{sec:scale_balancing_experiment,fig:scale_comparison} further shows a wider correct-order plateau on the synthetic data and earlier reference-order identification with lower reconstruction error on Swimmer.

The main contributions of this paper are summarized as follows:
\begin{enumerate}
\item \textbf{Model and structural analysis:} We formulate model-order selection as paired column-cardinality regularization within a fixed-dimensional NMF model and characterize its critical points and local minimizers. A sufficient-condition rank-recovery result is included as a supporting theoretical result, with its proof deferred to Appendix~\ref{app:proof_exact_recovery}.

\item \textbf{Algorithms and convergence analysis:} We develop an inertial proximal alternating linearized minimization method, its scale-balanced variant, and a proximal active-set method that refines the identified support through lower-dimensional nonnegative least-squares subproblems. Under explicit step-size and inertial-parameter conditions, a Kurdyka--\L{}ojasiewicz analysis proves finite length and whole-sequence convergence to critical points, including the scale-balancing operation.

\item \textbf{Automatic regularization path and signal experiments:} We construct a warm-started $\lambda$-path and a marginal-loss rule that select the active order without resizing the factors. The continuous-path comparison in \Cref{sec:lambda_path_analysis,fig:time_comparison} demonstrates its computational advantage over increasing- and decreasing-order discrete $r$-paths, while the ablation in \Cref{sec:scale_balancing_experiment,fig:scale_comparison} demonstrates the improved stability provided by scale balancing. Further experiments on synthetic data and diverse signal benchmarks evaluate order selection, reconstruction or factor recovery, and computational cost against singular-value hard thresholding and cross-validation.
\end{enumerate}

The remainder of this paper is organized as follows. Section~II introduces the variational tools and nonnegative-rank quantities used in the analysis. Section~III studies the structural properties of the proposed model. Sections~IV and V develop the proximal alternating and active-set algorithms and establish their convergence. Section~VI reports the numerical experiments, and Section~VII concludes the paper.

{\bf Notation:} $\mathbb{R}^{n\times m}$ denotes the vector space of all real $n\times m$ matrices, and $\mathbb{R}_+^{n\times m}$ denotes the set of non-negative real $n\times m$ matrices. The space is equipped with the trace inner product $\langle X,Y\rangle={\rm trace}(X^\top Y)$ and its induced Frobenius norm $\|X\|_F=\sqrt{\langle X,X\rangle}$. For a matrix $X$, $X^\top$ represents its transpose. We use $\sigma_i(X)$ to denote the $i$-th largest singular value of $X$, and $\sigma_{\min}(X)$ to denote its minimum non-zero singular value. The symbols $\|X\|_2$, $\|X\|_*$, and $\|X\|_{2,0}$ represent the spectral norm, the nuclear norm, and the column $\ell_{2,0}$-norm (number of non-zero columns) of matrix $X$, respectively. For a vector $x$, $\|x\|_2$ denotes its Euclidean $\ell_2$-norm. We use $X_i$ or $X_{:,i}$ to denote the $i$-th column of $X$, and $X_{i,:}$ to denote its $i$-th row. $J_{\!X}$ and $\overline{J}_{\!X}$ denote the index sets of the non-zero and zero columns in $X$, respectively. The symbol $\odot$ denotes the element-wise (Hadamard) product, and $[X]_+ = \max(X, 0)$ denotes element-wise non-negative truncation. Given a point $(\overline{U},\overline{V})\in\mathbb{R}^{n\times r}\times\mathbb{R}^{m\times r}$ and a constant $\delta>0$, let $\mathbb{B}_{\delta}(\overline{U},\overline{V}):=\{(U,V)\,|\,\|(U,V)-(\overline{U},\overline{V})\|_F\le\delta\}$. In the remainder of this paper, for a function $F(U,V)$, $\nabla_{\!U}F(U,V)$ and $\nabla_{\!V}F(U,V)$ denote the partial gradients of $F$ with respect to variables $U$ and $V$. Similarly, $\partial_1\Phi_{\lambda,\mu}(U',V')$ and $\partial_2\Phi_{\lambda,\mu}(U',V')$ denote the partial subdifferentials of $\Phi_{\lambda,\mu}$ at $(U',V')$ with respect to variables $U$ and $V$, respectively.

\section{Preliminaries}\label{Preliminaries}

We use standard notations from variational analysis. For an extended-real-valued function $h\!:\mathbb{R}^p\to[-\infty,+\infty]$, its regular, (limiting), and horizon subdifferentials, as well as its subderivative function, are denoted by $\widehat{\partial}h(x)$, $\partial h(x)$, $\partial^{\infty}h(x)$, and $dh(x)(\cdot)$, respectively, with detailed definitions available in \cite[Definition. 8.1, 8.3]{rockafellar1998variational}. A point $x \in \mathbb{R}^p$ satisfying $0\in\partial h(x)$ is a critical point ($x \in {\rm crit}\,h$).

To handle non-negativity, we introduce the indicator function $\delta_{\mathbb{R}_+}(W)$, which equals $0$ if $W \geq 0$ element-wise, and $\infty$ otherwise. We define proper, lower semi-continuous penalty functions for $U$ and $V$ as $g_1(U) := \lambda\norm[2,0]{U} + \delta_{\mathbb{R}_+}(U)$ and $g_2(V) := \lambda\norm[2,0]{V} + \delta_{\mathbb{R}_+}(V)$.
To encapsulate the smooth coupling terms of the objective, we define $F_\mu(U,V)$ as:
\begin{equation}
    F_{\mu}(U,V) := \frac{1}{2}\|X-UV^\top\|_F^2 + \frac{\mu}{2}\left( \|U\|_F^2 + \|V\|_F^2 \right).
\end{equation}
Consequently, the proposed model \cref{eq:l20_regularized_nmf} and its smooth auxiliary counterpart can be recast as unconstrained optimization problems:
\begin{align}
\min_{U, V} \Psi_{\lambda,\mu}(U,V) &:= F_{\mu}(U,V) + g_1(U) + g_2(V), \label{unconstrained_l20} \\
\min_{U, V} \widetilde{F}_{\mu}(U,V) &:= F_{\mu}(U,V) + \delta_{\mathbb{R}_+}(U) + \delta_{\mathbb{R}_+}(V). 
\label{unconstrained_group}
\end{align}

\noindent\textbf{Subdifferentials of the Objective.} 
The subdifferentials of the column $\ell_{2,0}$-norm have been fully characterized in \cite[Lemma 2.3]{tao2022column}. Combining these results with standard variational calculus \cite{rockafellar1998variational}, we characterize the subdifferential of $\Psi_{\lambda,\mu}$ as follows.
\begin{proposition}\label{prop:partial}
The partial subdifferentials of $\Psi_{\lambda,\mu}(U,V)$ evaluated at $(\overline{U},\overline{V})$ are given by: 
\begin{align*}
\partial_U \Psi_{\lambda,\mu} &= \left\{ G \;\middle|\; 
  \begin{aligned}
    G_j \in {} & \left(\overline{U}\overline{V}^\top-X\right)\overline{V}_j + \mu \overline{U}_j \\
    & + \partial \delta_{\mathbb{R}_+}(\overline{U}_j), \quad \forall j \in \Jcal{\overline{U}}
  \end{aligned}
\right\}, \\
\partial_V \Psi_{\lambda,\mu} &= \left\{ H \;\middle|\; 
  \begin{aligned}
    H_j \in {} & \left(\overline{U}\overline{V}^\top-X\right)^\top \overline{U}_j + \mu \overline{V}_j \\
    & + \partial \delta_{\mathbb{R}_+}(\overline{V}_j), \quad \forall j \in \Jcal{\overline{V}}
  \end{aligned}
\right\},
\end{align*}

where $\Jcal{\overline{U}}$ and $\Jcal{\overline{V}}$ denote the index sets of the non-zero columns for $\overline{U}$ and $\overline{V}$.
\end{proposition}

\noindent\textbf{Kurdyka-\L{}ojasiewicz (KL) Property.} 
A proper lower semi-continuous function $f$ satisfies the KL property at $\bar{x} \in \dom \partial f$ if there exists a neighborhood $U$ of $\bar{x}$ and a concave desingularizing function $\varphi$ such that the KL inequality $\varphi'(f(x) - f(\bar{x})) \dist(0, \partial f(x)) \ge 1$ holds \cite{bolte2014proximal}. Since the $\ell_0$-norm, the $\ell_2$-norm, and the indicator function $\delta_{\mathbb{R}_+}$ are semialgebraic, the objective $\Psi_{\lambda,\mu}$ inherently possesses semialgebraic properties, making it a KL function \cite[Appendix]{bolte2014proximal}.

\noindent\textbf{Nonnegative Rank and Nuclear Norm.} 
The nonnegative rank $\rank_+(A)$ is the smallest integer $r$ such that a nonnegative matrix $A$ can be factored as $A = UV$ with $U, V \geq 0$, satisfying $\rank(A) \leq \rank_+(A) \leq \min(n, m)$. We utilize the nonnegative nuclear norm $\nu_+(A)$ \cite[Theorem 1, Theorem 4]{fawzi2015lower}, which bounds the nonnegative rank via $\rank_+(A) \geq (\nu_+(A)/\|A\|_F)^2$, and has the variational form $\nu_+(A) = \min\{ \frac12\sum_{i} (\|u_i\|_2^2 +\|v_i\|_2^2) : A = \sum_{i} u_i v_i^\top, u_i, v_i \geq 0 \}$.

\begin{theorem}\label{thm:nmf_norm_bound}
For nonnegative matrices $X = UV^T$, if $\|U_{:,k}\|_2 = \|V_{:,k}\|_2$ for all $k = 1, \dots, r$, then: 
\begin{displaymath}
\norm[*]{X} \le \nu_+(X)\le\frac{1}{2}(\|U\|_F^2 + \|V\|_F^2) \le \sqrt{r}\|X\|_F.     
\end{displaymath}
\end{theorem}
The proof is given in Appendix~\ref{app:proof_nmf_norm_bound}.

\section{Structural Properties of the Proposed Model}

Although column $\ell_{2,0}$ regularization provides a direct mechanism for rank adaptation \cite{tao2022column}, its extension to NMF must account for the geometry of the nonnegative orthant and the relationship between active factor columns and the nonnegative rank of their product. We therefore characterize the relationship between the nonsmooth model $\Psi_{\lambda,\mu}$ and the smooth auxiliary model $\widetilde{F}_{\mu}$, followed by a supporting sufficient-condition result for rank recovery.

\begin{proposition}\label{prop:crit}
For any $\lambda > 0$ and $\mu > 0$, the critical point sets of $\Psi_{\lambda,\mu}$ and $\widetilde{F}_{\mu}$ are identical, i.e., $\mathrm{crit}\Psi_{\lambda,\mu} = \mathrm{crit}\widetilde{F}_{\mu}$. Furthermore, if $(\overline{U}^*,\overline{V}^*)$ is a global minimizer of $\Psi_{\lambda,\mu}$ satisfying the condition $\frac{2\lambda}{\mu} \ge \sqrt{\mathrm{rank}_+(\overline{U}^*(\overline{V}^*)^\top)}\norm[F]{\overline{U}^*(\overline{V}^*)^\top}$, it necessarily holds that:
\begin{displaymath}
\mathrm{rank}_+(\overline{U}^*(\overline{V}^*)^\top) = \norm[2,0]{\overline{U}^*} = \norm[2,0]{\overline{V}^*}.
\end{displaymath}
\end{proposition}

The proof is given in Appendix~\ref{app:structural}. Although incorporating the squared Frobenius penalty reduces scaling ambiguities, it inherently introduces a residual bias %
. The following proposition bounds this discrepancy from below to illustrate the influence of the perturbation parameter $\mu$ %
.

\begin{proposition}\label{prop:lower bound}
Suppose the observed matrix is $X = X^* + N$, where $X^*$ represents the true signal and $N$ represents additive noise. For any non-zero critical point $(\overline{U}, \overline{V}) \in \mathrm{crit}\widetilde{F}_\mu$, the reconstruction error satisfies the lower bound:
\begin{displaymath}
\norm[F]{\overline{U}\overline{V}^\top - X^*} \ge \max(0, \mu - \norm[F]{N}).
\end{displaymath}
\end{proposition}
The proof is given in Appendix~\ref{app:structural}. Per Proposition \ref{prop:lower bound}, relying solely on the smooth objective $\widetilde{F}_\mu$ with a large $\mu$ amplifies the reconstruction error and creates a structural bottleneck. The non-smooth model $\Psi_{\lambda,\mu}$ avoids the same error floor by allowing $\mu$ to be chosen sufficiently small while component selection is controlled by the column $\ell_{2,0}$ penalty.

\begin{proposition}[Sufficient Conditions for Exact Model-Order Recovery]\label{prop:exact_recovery}
Let $X^*$ be the true underlying non-negative matrix, observed as $X = X^* + N$. Assume the non-negative rank of $X^*$ coincides with its algebraic rank, i.e., $\mathrm{rank}_+(X^*) = \mathrm{rank}(X^*) := r^*$, and the magnitude of the noise is bounded by $\norm[F]{N} \le \frac{\sigma_{r^*}(X^*)}{5}$. For any regularization parameter $\lambda$ satisfying
\begin{displaymath}
\lambda \in \left[ \frac{1}{2} \norm[F]{N}^2, \frac{1}{32}\left( \sigma_{r^*}(X^*) - \norm[F]{N} \right)^2 \right],
\end{displaymath}
and a sufficiently small $\mu > 0$, any critical point $(\overline{U},\overline{V})$ of the model $\Psi_{\lambda,\mu}$ that satisfies the energy boundedness condition 
\begin{displaymath}
\Psi_{\lambda,\mu}(\overline{U},\overline{V}) \le \frac{1}{2} \norm[F]{N}^2 + \mu \sqrt{r^*}\norm[F]{X^*} + 2\lambda r^*,
\end{displaymath}
exactly recovers the true rank, yielding $\mathrm{rank}_+(\overline{U}\overline{V}^\top) = r^*$.
\end{proposition}
\begin{remark}
The condition $\mathrm{rank}_+(X^*) = \mathrm{rank}(X^*)$ is a standard assumption in the identifiability of exact NMF. As noted in \cite[Chapter 4.2]{gillis2020nonnegative}, this equality naturally holds in most practical data analysis applications, making it a well-justified premise for our exact recovery guarantees.
\end{remark}
The proof is given in Appendix~\ref{app:proof_exact_recovery}.

Finally, we characterize the topological correspondence between the local minimizers of the proposed non-smooth model $\Psi_{\lambda,\mu}$ and those of the smooth auxiliary model $\widetilde{F}_\mu$ %
.

\begin{proposition}\label{prop:local-minimizer}
Let $\lambda>0$ and $\mu>0$. If $(\overline{U},\overline{V})$ is a (strong) local minimizer of $\widetilde{F}_\mu$, then it concurrently serves as a (strong) local minimizer of $\Psi_{\lambda,\mu}$. Conversely, if $(\overline{U},\overline{V})$ is a non-zero (strong) local minimizer of $\Psi_{\lambda,\mu}$ with active support $J = \mathcal{J}(\overline{U})$, then the truncated submatrix pair $(\overline{U}_J,\overline{V}_J)$ is intrinsically a (strong) local minimizer of $\widetilde{F}_\mu$ restricted to the lower-dimensional space $\mathbb{R}^{n\times|J|}\times\mathbb{R}^{m\times|J|}$.
\end{proposition}

The proof of Proposition \ref{prop:local-minimizer}, which parallels \cite[Proposition 2.7]{tao2022column}, is given in Appendix~\ref{app:proof_local_minimizer}.

\section{iPALM Method for the \texorpdfstring{$\ell_{2,0}$}{L2,0}-Regularized NMF Model}\label{sec:iPALM}

\subsection{Standard iPALM Algorithm and Convergence Analysis}\label{subsec:standard_ipalm}
Fix any $(U,V)\in\mathbb{R}^{n\times r}\times\mathbb{R}^{m\times r}$. The block-gradients of $F_{\mu}$ and their respective Lipschitz constants are given by:
{\small
\begin{align*}
    \nabla_U F_{\mu}(U,V) &= (UV^\top - X)V + \mu U, \quad \tau_{\!V}\!:= \norm[2]{V^\top V} + \mu, \\
    \nabla_V F_{\mu}(U,V) &= (VU^\top - X^\top)U + \mu V, \quad \tau_{\!U}\!:= \norm[2]{U^\top U} + \mu.
\end{align*}
}%
By the descent lemma, we can majorize the smooth part $F_\mu$ and apply an inertial extrapolation step to accelerate the alternating minimization. This leads to the proposed iPALM method detailed in Algorithm \ref{alg:iPALM}.

\begin{algorithm}[!t]
 \caption{\label{alg:iPALM}{\bf iPALM Method for solving \cref{eq:l20_regularized_nmf}}}
 \textbf{Initialization:} 
 Choose $(U^0,V^0) \ge 0$. Set $(U^{-1},V^{-1}) := (U^0,V^0)$. Choose bounds $\bar{\beta} \in [0,1)$, initial $\beta_0 \in [0, \bar{\beta}]$, and $0 < \alpha_1 \le \alpha_2$. Set $k:=0$.\\
\textbf{while} stopping conditions are not satisfied \textbf{do}
 \begin{itemize}[leftmargin=*, nosep]
  \item[\bf 1.] \textbf{Update $U$:} Let $\tau_{\!V^k} = \norm[2]{(V^k)^\top V^k} + \mu$, select $\gamma_{1,k}\in\tau_{\!V^{k}}+{[\alpha_1,\alpha_2]}$. \\
                Compute extrapolated $\widetilde{U}^k := U^k+\beta_k(U^k-U^{k-1})$ and update:
                \begin{multline}\label{Uk-subprob}
                 U^{k+1} \in \mathop{\arg\min}_{U\in\mathbb{R}^{n\times r}} \Big\{\langle\nabla_U F_{\mu}(\widetilde{U}^k,V^k), U\rangle \\+ \frac{\gamma_{1,k}}{2}\norm[F]{U\!-\!\widetilde{U}^k}^2 + g_1(U)\Big\}.
                \end{multline}

  \item[\bf 2.] \textbf{Update $V$:} Let $\tau_{\!U^{k+1}} = \norm[2]{(U^{k+1})^\top U^{k+1}} + \mu$, select $\gamma_{2,k}\in\tau_{\!U^{k+1}}+{[\alpha_1,\alpha_2]}$. \\
                Compute extrapolated $\widetilde{V}^k := V^k+\beta_k(V^k-V^{k-1})$ and update:
                \begin{multline}\label{Vk-subprob}
                 V^{k+1} \in \mathop{\arg\min}_{\!V\in\mathbb{R}^{m\times r}} \Big\{\langle\nabla_V F_{\mu}(U^{k+1},\widetilde{V}^k), V\rangle \\+ \frac{\gamma_{2,k}}{2}\norm[F]{V\!-\!\widetilde{V}^k}^2 + g_2(V)\Big\}.
                \end{multline}

  \item[\bf 3.] Choose inertial parameter $\beta_{k+1} \in [0, \bar{\beta}]$. Let $k\leftarrow k+1$.
 \end{itemize}
 \textbf{end while}
\end{algorithm}
\begin{remark}\label{remark1-iPALM}
 The subproblems \eqref{Uk-subprob} and \eqref{Vk-subprob} admit efficient closed-form solutions by combining an element-wise non-negativity truncation with a column-wise hard thresholding for the $\ell_{2,0}$-norm. Let $P^k = \widetilde{U}^k - \frac{1}{\gamma_{1,k}}\nabla_U F_{\mu}(\widetilde{U}^k,V^k)$ and $Q^k = \widetilde{V}^k - \frac{1}{\gamma_{2,k}}\nabla_V F_{\mu}(U^{k+1},\widetilde{V}^k)$ be the gradient descent updates at the extrapolated points. Denoting the element-wise maximums as $\hat{P}^k = \max\{P^k, 0\}$ and $\hat{Q}^k = \max\{Q^k, 0\}$, the $i$-th columns ($i=1,\ldots,r$) of the updates are explicitly given by:
 \begin{align*}
    U_i^{k+1} &= \begin{cases} 
      \hat{P}_i^k, & \text{if } \norm[2]{\hat{P}_i^k} > \sqrt{\frac{2\lambda}{\gamma_{1,k}}}, \\ 
      0, & \text{otherwise},
    \end{cases} \\
    V_i^{k+1} &= \begin{cases} 
      \hat{Q}_i^k, & \text{if } \norm[2]{\hat{Q}_i^k} > \sqrt{\frac{2\lambda}{\gamma_{2,k}}}, \\ 
      0, & \text{otherwise}.
    \end{cases}
 \end{align*}
 This decoupled, closed-form solution ensures that the subproblems can be solved in linear time, maintaining the computational efficiency of the algorithm.
\end{remark}

To establish the global convergence of Algorithm \ref{alg:iPALM} under the Kurdyka-{\L}ojasiewicz (KL) framework, we first introduce several auxiliary parameters. For any iteration $k \in \mathbb{N}$, we define:
\begin{equation}\label{eq:alpha_defs}
    \alpha_{1,k} := \gamma_{1,k} - \tau_{V^{k}} \quad \text{and} \quad \alpha_{2,k} := \gamma_{2,k} - \tau_{U^{k+1}}.
\end{equation}

Central to our convergence analysis is the construction of a surrogate potential function. For suitably chosen constants $\rho_1, \rho_2 \in \left(0, \frac{\alpha_1}{2\alpha_2}\right)$, the potential function $\Xi$ is formulated as follows:
{\small
\begin{equation}\label{eq:potential_function}
    \Xi(U,V,U',V') := \Psi(U,V) + \frac{\rho_1\alpha_2}{2}\|U-U'\|_F^2 + \frac{\rho_2\alpha_2}{2}\|V-V'\|_F^2.
\end{equation}
}%
The subsequent proposition and its proof framework build on \cite{tao2022column,pock2016inertial}; the details are given in Appendix~\ref{app:proof_kl_ipalm}.

\begin{proposition}\label{prop:KL_properties-iPALM}
Let $\{(U^k,V^k)\}_{k\in\mathbb{N}}$ be the sequence generated by Algorithm \ref{alg:iPALM}, and let $\tau := \sup_{k\in\mathbb{N}}\max(\tau_{V^k},\tau_{U^{k+1}}) < \infty$. Assume that the inertial parameters $\beta_k$ are chosen such that $\beta_k \in [0, \hat{\beta}]$, where the uniform upper bound $\hat{\beta}$ satisfies:
{\small
\begin{equation*}
    0 \le \hat{\beta} < \min\left(\sqrt{\frac{\rho_1(1-\rho_1)\alpha_2}{(1-\rho_1)\tau+\alpha_2}}, \sqrt{\frac{\rho_2(1-\rho_2)\alpha_2}{(1-\rho_2)\tau+\alpha_2}}\right).
\end{equation*}
}%
Then, the following statements hold:
\begin{itemize}
    \item [(i)] \textbf{Sufficient Decrease:} The sequence of potential function values strictly decreases. Specifically, defining the non-negative coefficients
    {\small
    \begin{align*}
        \nu_{1,k} &:= \frac{(\alpha_{1,k}-\rho_1\alpha_2)(\rho_1\alpha_2-\tau_{V^{k}}\beta_k^2)-\alpha_{1,k}^2\beta_k^2}{2(\alpha_{1,k}-\rho_1\alpha_2)}, \\
        \nu_{2,k} &:= \frac{(\alpha_{2,k}-\rho_2\alpha_2)(\rho_2\alpha_2-\tau_{U^{k+1}}\beta_k^2)-\alpha_{2,k}^2\beta_k^2}{2(\alpha_{2,k}-\rho_2\alpha_2)},
    \end{align*}
    }%
the following inequality is satisfied for each $k \in \mathbb{N}$:
\begin{equation}\label{eq:sufficient_decrease}
\begin{multlined}
        \Xi(U^{k+1},V^{k+1},U^k,V^k) - \Xi(U^k,V^k,U^{k-1},V^{k-1}) \\
        \le -\nu_{1,k}\|U^{k}-U^{k-1}\|_F^2 - \nu_{2,k}\|V^k-V^{k-1}\|_F^2.
    \end{multlined}    
\end{equation}

    \item [(ii)] \textbf{Boundedness and Accumulation:} The generated sequence $\{(U^k,V^k)\}_{k\in\mathbb{N}}$ is uniformly bounded. Consequently, its set of accumulation points, denoted by $\Upsilon$, is nonempty and compact. Furthermore, the sequence $\{\Xi(U^{k},V^{k},U^{k-1},V^{k-1})\}_{k\in\mathbb{N}}$ monotonically converges to a finite limit $\varpi^*$, with $\Xi \equiv \varpi^*$ on the entire set $\Upsilon$.

    \item [(iii)] \textbf{Subgradient Bound:} There exist positive constants $c_1$ and $c_2$ bounding the subgradient such that for each $k \in \mathbb{N}$:
    \begin{multline}\label{eq:subgradient_bound}
        {\rm dist}\big(0,\partial\Xi_{\lambda,\mu}(U^{k+1},V^{k+1},U^k,V^k)\big)\\
        \begin{aligned}
            &\leq c_1\big(\|U^{k+1}-U^{k}\|_F + \|U^k-U^{k-1}\|_F\big) \\
        &\quad + c_2\big(\|V^{k+1}-V^{k}\|_F + \|V^k-V^{k-1}\|_F\big). 
        \end{aligned}
    \end{multline}
    Here, the constants are given by $c_1 := \tau+\overline{\gamma}+2\rho_1\alpha_2$ and $c_2 := c_f+2\tau+\overline{\gamma}+2\rho_2\alpha_2$, where we denote $c_f := \sup_{k\in\mathbb{N}}\{\|U^k(V^k)^{\top}-X\|\}$.
\end{itemize}
\end{proposition}

To tailor the analysis to our algorithms, we adapt the framework of \cite[Theorem 1]{bolte2014proximal} and \cite[Theorem 3.1]{liu_refined_2019} by introducing an intermediate residual $\mathcal{R}^k$, leading to the following unified convergence result; see Appendix~\ref{app:proof_unified_convergence}.

\begin{theorem}[Unified Global Convergence]\label{thm:unified_global_convergence}
Let $\{X^k\}$ be a bounded sequence, and assume that $\psi$ is a KL function whose value is finite and constant on the compact cluster set of $\{X^k\}$. If there exist constants $a, b, c > 0$ such that the sequence satisfies: 
(i) $\psi(X^k) - \psi(X^{k+1}) \ge a (\mathcal{R}^k)^2$; 
(ii) $\|W^{k+1}\|_F \le b (\mathcal{R}^k + \mathcal{R}^{k+1})$ for $W^{k+1} \in \partial \psi(X^{k+1})$; and 
(iii) $\|\Delta Z^k\|_F \le c \mathcal{R}^k$, 
then $\{Z^k\}$ has finite length and converges to a single critical point.
\end{theorem}
By invoking Theorem~\ref{thm:unified_global_convergence} with $Z^k = (U^k,V^k)$, $X^k = (Z^k, Z^{k-1})$, $\psi(X^k) = \Xi(U^{k},V^{k},U^{k-1},V^{k-1})$, $\Delta Z^k = Z^k - Z^{k-1}$ and $\mathcal{R}^k = \|U^k-U^{k-1}\|_F + \|V^k-V^{k-1}\|_F$, the global convergence of \cref{alg:iPALM} to a critical point $Z^*$ is immediately established.

\subsection{Scale-Balanced iPALM: Motivation and Algorithm Design}\label{subsec:scale_balanced_ipalm}
To establish a profound motivation for our proposed algorithmic framework, it is crucial to analyze the intrinsic numerical behavior of the primary data fidelity term $\frac{1}{2}\|X-UV^T\|_F^2$. As formally demonstrated in classical non-negative matrix factorization (NMF) paradigms \cite{gillis2020nonnegative}, the reconstruction error exhibits an inherent scale-invariance with respect to the rank-one factors. Specifically, for any positive diagonal scaling matrix $\Lambda = \text{diag}(\alpha_1, \dots, \alpha_r)$ with $\alpha_k > 0$, replacing $U$ with $U\Lambda$ and $V$ with $V\Lambda^{-1}$ leaves the matrix product $UV^T$ completely unaltered. However, first-order optimization schemes—such as proximal gradient methods and alternating linearized algorithms—are notoriously sensitive to the spectral properties induced by such scaling.

In an alternating update framework, when solving the subproblem with respect to $V$ while keeping $U$ fixed, the local curvature and the corresponding Lipschitz constant of the gradient are directly dictated by the eigenvalues of the matrix $U^TU$. Symmetrically, the conditioning of the subproblem for updating $U$ depends heavily on $V^TV$. If the columns of $U$ or $V$ suffer from severe scale imbalance—meaning certain columns possess disproportionately larger $\ell_2$-norms than others—the proxy matrices $U^TU$ and $V^TV$ become highly ill-conditioned. This spectral distortion forces the gradient steps to become heavily mismatched, resulting in severely sluggish convergence or numerical stagnation, even when the overall reconstruction error appears theoretically stable.

To eliminate this numerical vulnerability embedded within the data fidelity term, we introduce an explicit scale balancing ($L_2$ normalization) step into the iteration sequence. By strictly enforcing $\|U_{:,k}\|_2 = \|V_{:,k}\|_2$ for each $k = 1, \dots, r$, we dynamically rectify the ill-posed scaling of the rank-one components. This operation substantially improves the conditioning of the underlying subproblems without altering the reconstruction fidelity or manipulating the support structure governed by the $\ell_{2,0}$-norm. Incorporating this crucial stabilization mechanism, we present our novel scale-balanced algorithm, as detailed in Algorithm \ref{alg:iPALM_scale}.
\begin{algorithm}[!t]
 \caption{\label{alg:iPALM_scale}{\bf Scale-Balanced iPALM Method for solving \cref{eq:l20_regularized_nmf}}}
 \textbf{Initialization:} Same as Algorithm \ref{alg:iPALM}.\\
\textbf{while} stopping conditions are not satisfied \textbf{do}
 \begin{itemize}[leftmargin=*, nosep]
  \item[\bf 1.] \textbf{Extrapolate \& Update:} Compute intermediate variables $\hat{U}^{k+1}$ and $\hat{V}^{k+1}$ by executing Steps 1 and 2 of Algorithm \ref{alg:iPALM} respectively.
  
  \item[\bf 2.] \textbf{Scale Balancing:} For each column $i = 1, \dots, r$:
                \begin{itemize}[leftmargin=1.5em, nosep]
                    \item \textbf{if} $\|\hat{U}_i^{k+1}\|_2 \|\hat{V}_i^{k+1}\|_2 = 0$, set $U_i^{k+1} = \mathbf{0}$ and $V_i^{k+1} = \mathbf{0}$;
                    \item \textbf{otherwise}, compute $S^k_{ii} = \sqrt{\|\hat{V}_i^{k+1}\|_2 / \|\hat{U}_i^{k+1}\|_2}$ and update:
                    \begin{equation*}
                        U_i^{k+1} \leftarrow \hat{U}_i^{k+1}S^k_{ii}, \quad V_i^{k+1} \leftarrow \hat{V}_i^{k+1}(S^k_{ii})^{-1}.
                    \end{equation*}
                \end{itemize}

  \item[\bf 3.] Choose inertial parameter $\beta_{k+1} \in [0, \bar{\beta}]$. Let $k\leftarrow k+1$.
 \end{itemize}
 \textbf{end while}
\end{algorithm}
\begin{remark}\label{remark2-scale}

{\bf(a) Scaling Optimality:} By the AM-GM inequality, the structural penalty $\norm[F]{U}^2 + \norm[F]{V}^2$ is globally minimized when $\norm[2]{U_i} = \norm[2]{V_i}, \forall i$. The scaling matrix $S^k$ strictly enforces this equilibrium without altering the data fidelity product $UV^\top$ or the $\ell_{2,0}$ sparsity.
 
 {\bf(b) Inertial Consistency:} Retrospectively scaling historical variables $(U^k, V^k)$ would inherently disrupt their already-achieved equilibrium and artificially distort the momentum trajectory. Applying $S^k$ exclusively to current variables preserves rigorous extrapolation dynamics for convergence analysis.

 {\bf(c)} {\bf Strict Co-Sparsity.} When either $\hat{U}_i^{k+1}$ or $\hat{V}_i^{k+1}$ is a zero vector, the rank-1 component $\hat{U}_i^{k+1} (\hat{V}_i^{k+1})^\top$ inherently evaluates to zero. Setting both columns to zero explicitly prevents division by zero in the scaling step and eliminates redundant Frobenius penalties without altering the data fidelity term.
\end{remark}

\subsection{Theoretical Properties and Global Convergence of Scale-Balanced iPALM}\label{subsec:convergence_scale_balanced}

Before establishing the global convergence under the Kurdyka-{\L}ojasiewicz (KL) framework, we explicitly analyze the mathematical properties of the Scale Balancing step introduced in Algorithm \ref{alg:iPALM_scale}. The following lemma provides an analytical bound on the physical displacement caused by the scaling step on a single column pair, rigorously covering the singular cases where the variables may be truncated to zero.

\begin{lemma}[Column-wise Bounded Jump of Scale Balancing]\label{lemma:column_jump}
Let $x \in \mathbb{R}^n$ and $y \in \mathbb{R}^m$ be any paired column vectors satisfying the balanced scale condition $\|x\|_2 = \|y\|_2 = \rho$. Suppose these vectors undergo a gradient update to intermediate states $\hat{x} = x + \Delta x$ and $\hat{y} = y + \Delta y$. Let $x^+$ and $y^+$ be the scale-balanced output vectors obtained according to Step 2 and Remark \ref{remark2-scale}(c) of Algorithm \ref{alg:iPALM_scale}. The total displacement is strictly bounded by:
\begin{displaymath}
    \|x^+ - \hat{x}\|_2 + \|y^+ - \hat{y}\|_2 \le \|\Delta x\|_2 + \|\Delta y\|_2.
\end{displaymath}
\end{lemma}

With the column-wise properties established, we generalize these bounds to the entire matrix structure. The following lemma consolidates the structural deviations introduced by the normalization step into two forms: squared and non-squared Frobenius norms.

\begin{lemma}[Matrix-wise Displacement Bounds]\label{lemma:matrix_jump}
Assuming the matrices from the previous iteration are well-balanced ($\|U_{:,i}^k\|_2 = \|V_{:,i}^k\|_2$ for $i = 1, \dots, r$), the total displacement caused by the scaling step across the entire matrices is bounded by:
\begin{align}
 \begin{multlined}
   \|U^{k+1} - \hat{U}^{k+1}\|_F^2 + \|V^{k+1} - \hat{V}^{k+1}\|_F^2 \\\le 2 \left( \|\hat{U}^{k+1} - U^k\|_F^2 + \|\hat{V}^{k+1} - V^k\|_F^2 \right), \label{eq:matrix_jump_sq}  
\end{multlined} \\
\begin{multlined}
     \|U^{k+1} - \hat{U}^{k+1}\|_F + \|V^{k+1} - \hat{V}^{k+1}\|_F \\\le \sqrt{2} \left( \|\hat{U}^{k+1} - U^k\|_F + \|\hat{V}^{k+1} - V^k\|_F \right). \label{eq:matrix_jump_l1}
\end{multlined}
\end{align}
\end{lemma}

Proofs of \Cref{lemma:column_jump,lemma:matrix_jump} are collected in Appendix~\ref{app:balancing}. Combining these lemmas with \Cref{prop:KL_properties-iPALM} yields \Cref{prop:iPALM_scale}; its proof is given in Appendix~\ref{app:proof_scale_ipalm}.
\begin{proposition}\label{prop:iPALM_scale}
Let $\{(U^k,V^k)\}_{k\in\mathbb{N}}$ be the sequence generated by Algorithm \ref{alg:iPALM_scale}, and let $\tau := \sup_{k\in\mathbb{N}}\max(\tau_{V^k},\tau_{\hat{U}^{k+1}}) < \infty$. For brevity, setting $\rho=\rho_1=\rho_2\in(0,\frac{\alpha_1}{6\alpha_2})$, assume that the inertial parameters $\beta_k$ are chosen such that $\beta_k \in [0, \widetilde{\beta}]$, where the uniform upper bound $\widetilde{\beta}$ satisfies:
\begin{equation*}
    0 \le \widetilde{\beta} < \sqrt{\frac{\rho(1-\rho)\alpha_2}{(1-\rho)\tau+\alpha_2}}.
\end{equation*}
Then, the following statements hold:
\begin{itemize}
    \item [(i)] \textbf{Sufficient Decrease:} The sequence of potential function values strictly decreases. Specifically, defining the non-negative coefficients
    {\small
    \begin{align*}
        \nu_{1,k} &:= \frac{(\alpha_{1,k}-3\rho\alpha_2)(3\rho\alpha_2-\tau_{V^{k}}\beta_k^2)-\alpha_{1,k}^2\beta_k^2}{2(\alpha_{1,k}-3\rho\alpha_2)}, \\
        \nu_{2,k} &:= \frac{(\alpha_{2,k}-3\rho\alpha_2)(3\rho\alpha_2-\tau_{U^{k+1}}\beta_k^2)-\alpha_{2,k}^2\beta_k^2}{2(\alpha_{2,k}-3\rho\alpha_2)},
    \end{align*}
    }%
the following inequality is satisfied for each $k\in\mathbb{N}$:
    \begin{multline}\label{eq:pro3-diff}
        \Xi(U^{k+1},V^{k+1},U^k,V^k) - \Xi(U^k,V^k,U^{k-1},V^{k-1})  \\
        \le -\nu_{1,k}\|U^{k}-U^{k-1}\|_F^2 - \nu_{2,k}\|V^k-V^{k-1}\|_F^2.
    \end{multline}

    \item [(ii)] \textbf{Boundedness and Accumulation:} The generated sequence $\{(U^k,V^k)\}_{k\in\mathbb{N}}$ is bounded. Consequently, the set of accumulation points of the sequence $\{(U^{k},V^{k},U^{k-1},V^{k-1})\}_{k\in\mathbb{N}}$, denoted by $\Upsilon$, is nonempty and compact. Furthermore, the sequence $\{\Xi(U^{k},V^{k},U^{k-1},V^{k-1})\}_{k\in\mathbb{N}}$ converges to a finite limit $\varpi^*$, with $\Xi \equiv \varpi^*$ on the entire set $\Upsilon$.

    \item [(iii)] \textbf{Subgradient Bound:} There exist positive constants $c_1$ and $c_2$ bounding the subgradient such that for each $k\in\mathbb{N}$:
    {\small
    \begin{multline*}
        {\rm dist}\big(0,\partial\Xi_{\lambda,\mu}(U^{k+1},V^{k+1},U^k,V^k)\big)\\
        \begin{aligned}
           &\leq c_1\big(\|U^{k+1}-\hat{U}^{k+1}\|_F+\|\hat{U}^{k+1}-U^{k}\|_F+\|U^k-U^{k-1}\|_F\big)\\
        &\quad+c_2\big(\|V^{k+1}-\hat{V}^{k+1}\|_F+\|\hat{V}^{k+1}-V^{k}\|_F+\|V^k-V^{k-1}\|_F\big).  
        \end{aligned}    
    \end{multline*}
    }%
Here, the constants are given by $c_1 := \tau+\overline{\gamma}+2\rho\alpha_2+c_f$ and $c_2 := c_f+2\tau+\overline{\gamma}+2\rho\alpha_2$, where we denote $c_f := \sup_{k\in\mathbb{N}}\{\|U^k(V^k)^{\top}-X\|\}$.
\end{itemize}
\end{proposition}

By invoking Theorem~\ref{thm:unified_global_convergence} with $Z^k = (U^k,V^k)$, $X^k = (Z^k, Z^{k-1})$, $\psi(X^k) = \Xi(U^{k},V^{k},U^{k-1},V^{k-1})$, $\Delta Z^k = Z^k - Z^{k-1}$ and $\mathcal{R}^k = \|U^{k}-\hat{U}^{k}\|_F+\|\hat{U}^{k}-U^{k-1}\|_F + \|V^{k}-\hat{V}^{k}\|_F+\|\hat{V}^{k}-V^{k-1}\|_F$, the global convergence of \cref{alg:iPALM_scale} to a critical point $Z^*$ is immediately established.

\section{A Proximal Active-Set Method based on P-Stationarity}

To further accelerate convergence and rigorously tackle the exact $\ell_{2,0}$-norm sparsity under non-negativity constraints, we propose a proximal active-set method. This approach elegantly bridges the theoretical concept of Proximal-stationary (P-stationary) points with the computational efficiency of subspace optimization. 

\subsection{P-Stationary Point and Active Set Formulation}

 Let $F_{\mu}(U,V)$ denote the smooth coupling terms of the objective, and let $\gamma_{1,k} > 0$ denote the step-size parameter (typically bounded by the Lipschitz constant of $\nabla_U F_{\mu}$). A point $U$ satisfies the P-stationary condition \cite{zhang2021recursion,zhou2021newton} if:
\begin{equation}\label{eq:p_stationary}
    U = \prox_{g_1}^{1/\gamma_{1,k}}\left( U - \frac{1}{\gamma_{1,k}} \nabla_U F_{\mu}(U, V) \right).
\end{equation}

Analyzing this proximal map column by column gives the following necessary and sufficient conditions; the short proof is included in Appendix~\ref{app:pstationarity}.

\begin{theorem}[P-Stationarity Conditions]\label{thm:p_stationarity}
Given the penalty parameter $\lambda > 0$ and step-size $\gamma_{1,k} > 0$, a matrix $U \in \mathbb{R}_+^{n \times r}$ is a P-stationary point satisfying \eqref{eq:p_stationary} if and only if for each column index $i \in \{1,\dots,r\}$, exactly one of the following two conditions holds:
\begin{enumerate}
    \item[\bf (a)] {\bf Inactive (Zero) Column:} $U_i = 0$, and the negative part of the gradient is bounded by the threshold:
    \begin{equation*}
        \norm[2]{\left[ -\nabla_U F_{\mu}(U, V)_i \right]_+} \leq \sqrt{2\lambda \gamma_{1,k}};
    \end{equation*}
    \item[\bf (b)] {\bf Active (Non-zero) Column:} $U_i \neq 0$, satisfying the non-negative KKT conditions $U_i \geq 0$, $\nabla_U F_{\mu}(U, V)_i \geq 0$, and $U_i \odot \nabla_U F_{\mu}(U, V)_i = 0$, alongside the norm bound:
    \begin{equation*}
        \norm[2]{U_i} \geq \sqrt{\frac{2\lambda}{\gamma_{1,k}}}.
    \end{equation*}
\end{enumerate}
\end{theorem}
Theorem \ref{thm:p_stationarity} provides a powerful algorithmic foundation. It explicitly partitions the columns of $U$ into two disjoint sets based on a calculable threshold. To operationalize this, we define a projected gradient descent step prior to the $\ell_{2,0}$-norm hard thresholding:
\begin{displaymath}
    Z_U^k := \max\left(0, U^k - \frac{1}{\gamma_{1,k}} \nabla_U F_{\mu}(U^k, V^k)\right).
\end{displaymath}
Following the theoretical bounds established in the theorem, we can unequivocally identify the active set $\mathcal{T}_U^k$ at iteration $k$ via the thresholding rule:
\begin{displaymath}
    \mathcal{T}_U^k := \left\{ i \in \{1,\dots,r\} \;\Big|\; \norm[2]{(Z_U^k)_i} > \sqrt{\frac{2\lambda}{\gamma_{1,k}}} \right\}.
\end{displaymath}

This active set formulation completely decouples the optimization over $U$. For any column index in the inactive set ($i \notin \mathcal{T}_U^k$), we strictly enforce $U_i^{k+1} = 0$, aligning with Condition (a). For the active subspace ($i \in \mathcal{T}_U^k$), the $\ell_{2,0}$-norm penalty acts as a constant, and the objective reduces to minimizing the smooth function $F_\mu$ over the non-negative orthant to satisfy Condition (b). This is mathematically equivalent to approximately solving a Non-Negative Least Squares (NNLS) subproblem strictly confined to the active columns. 

Applying a symmetric derivation to the block variable $V$, we establish the proposed Active-Set algorithm, detailed in Algorithm \ref{alg:ActiveSet}.

\begin{algorithm}[!t]
 \caption{\label{alg:ActiveSet}{\bf Proximal Active-Set Method for \cref{eq:l20_regularized_nmf}}}
 \textbf{Initialization:} Choose $U^0 \ge 0$, $V^0 \ge 0$. Set $k:=0$. \\
 \textbf{while} stopping conditions are not satisfied \textbf{do}
 \begin{itemize}[leftmargin=*, nosep]
  \item[\bf 1.] \textbf{Update $U$:} Select $\gamma_{1,k}\in \norm[2]{(V^k)^\top V^k} + \mu + [\alpha_1,\alpha_2]$. Compute intermediate variable $\tilde{U}^{k+1}$:
                \begin{multline}\label{eq:pasm_tilde_U}
                 \tilde{U}^{k+1} \in \mathop{\arg\min}_{U\in\mathbb{R}^{n\times r}} \Big\{\langle\nabla_U F_{\mu}(U^k,V^k), U\rangle\\ + \frac{\gamma_{1,k}}{2}\norm[F]{U\!-\!U^k}^2 + g_1(U)\Big\}.
                \end{multline}
                Get active set $\mathcal{T}_U^k = \big\{ i \;\big|\; \tilde{U}_i^{k+1} \neq 0 \big\}$. Set $\hat{U}_{i \notin \mathcal{T}_U^k}^{k+1} = 0$. \\
                For $\mathcal{T} = \mathcal{T}_U^k$, approximately solve the NNLS subproblem:
                \begin{equation}\label{eq:NNLS_U}
                    \hat{U}_{\mathcal{T}}^{k+1} \approx \mathop{\arg\min}_{U_{\mathcal{T}} \ge 0} \Big\{ \frac{1}{2}\norm[F]{X - U_{\mathcal{T}}(V_{\mathcal{T}}^k)^\top}^2 + \frac{\mu}{2}\norm[F]{U_{\mathcal{T}}}^2 \Big\}.
                \end{equation}

  \item[\bf 2.] \textbf{Update $V$:} Select $\gamma_{2,k}\in \norm[2]{(\hat{U}^{k+1})^\top \hat{U}^{k+1}} + \mu + [\alpha_1,\alpha_2]$. Compute intermediate variable $\tilde{V}^{k+1}$:
                \begin{multline}\label{eq:pasm_tilde_V}
                 \tilde{V}^{k+1} \in \mathop{\arg\min}_{V\in\mathbb{R}^{m\times r}} \Big\{\langle\nabla_V F_{\mu}(\hat{U}^{k+1},V^k), V\rangle \\+ \frac{\gamma_{2,k}}{2}\norm[F]{V\!-\!V^k}^2 + g_2(V)\Big\}.
                \end{multline}
                Get active set $\mathcal{T}_V^k = \big\{ i \;\big|\; \tilde{V}_i^{k+1} \neq 0 \big\}$. Set $\hat{V}_{i \notin \mathcal{T}_V^k}^{k+1} = 0$. \\
                For $\mathcal{T} = \mathcal{T}_V^k$, approximately solve the NNLS subproblem:
                \begin{equation}\label{eq:NNLS_V}
                    \hat{V}_{\mathcal{T}}^{k+1} \approx \mathop{\arg\min}_{V_{\mathcal{T}} \ge 0} \Big\{ \frac{1}{2}\norm[F]{X^\top - V_{\mathcal{T}}(\hat{U}_{\mathcal{T}}^{k+1})^\top}^2 + \frac{\mu}{2}\norm[F]{V_{\mathcal{T}}}^2 \Big\}.
                \end{equation}

  \item[\bf 3.] \textbf{Scale Balancing:} Compute $S^{k}$ and update $U^{k+1}$, $V^{k+1}$ exactly as in Step 2 of Algorithm \ref{alg:iPALM_scale}. Let $k \leftarrow k+1$.
 \end{itemize}
 \textbf{end while}
\end{algorithm}
\begin{remark}\label{remark3-ASHALS}
 {\bf(a)} {\bf Equivalence to Subspace Optimization.} On the active set, finding the stationary point fundamentally isolates the optimization to the non-zero variables, explicitly reducing the dimensionality of the problem. Since the non-negativity constraint strictly dictates the KKT conditions (i.e., $\hat{U}_{\mathcal{T}}^{k+1} \odot \nabla_{U_{\mathcal{T}}} F_{\mu}(\hat{U}_{\mathcal{T}}^{k+1}, V^k) = 0$, $\hat{U}_{\mathcal{T}}^{k+1} \ge 0$, and $\nabla_{U_{\mathcal{T}}} F_{\mu}(\hat{U}_{\mathcal{T}}^{k+1}, V^k) \ge 0$), the subproblems \eqref{eq:NNLS_U} and \eqref{eq:NNLS_V} cleanly translate into NNLS optimization confined to the active support.
 
 {\bf(b)} {\bf Inexact Solving of the NNLS Subproblem.} Calculating the exact solution for the NNLS subproblems is computationally expensive and practically unnecessary for large-scale matrices. Instead, we seek an approximate solution to decrease the objective effectively. While various algorithms can be employed for this inexact solve---such as Multiplicative Updates (MU), Projected Gradient Method (PGM), or Interior-Point methods---we adopt the Hierarchical Alternating Least Squares (HALS) method in this study due to its superior empirical efficiency. HALS inexactly minimizes the objective function column-by-column. For instance, to solve the subproblem \eqref{eq:NNLS_U}, let $A = X V_{\mathcal{T}}^k$ and $B = (V_{\mathcal{T}}^k)^\top V_{\mathcal{T}}^k$. Initializing the active submatrix $\hat{U}_{\mathcal{T}}$ with the Stage 1 output $\tilde{U}_{\mathcal{T}}^{k+1}$, the HALS update rule for the $j$-th column is formulated as:
 {\small
 \begin{displaymath}
     (\hat{U}_{\mathcal{T}})_j \leftarrow \max\left(0, \frac{A_j - \sum_{l \neq j} (\hat{U}_{\mathcal{T}})_l B_{lj}}{B_{jj} + \mu} \right), \quad \text{for } j = 1, \dots, |\mathcal{T}|.
 \end{displaymath}
 }%
By iteratively updating each column, this element-wise division avoids the massive computational overhead associated with exact solvers, and the final converged state naturally yields the desired Stage 2 variable $\hat{U}_{\mathcal{T}}^{k+1}$ with remarkably low per-iteration cost.
\end{remark}

\subsection{Global Convergence Analysis}
\label{subsec:pasm_convergence}

To rigorously establish the global convergence of the Proximal Active-Set Method (\Cref{alg:ActiveSet}) under the Kurdyka-{\L}ojasiewicz (KL) framework, we first analyze the properties of the single-sided HALS update for $U$. This step sequentially minimizes $\mathcal{L}(U) := h(U) + \delta_{\ge 0}(U)$ with $h(U) = \frac{1}{2}\|X-UV^\top\|_F^2 + \frac{\mu}{2}\|U\|_F^2$ while $V$ is fixed, yielding the exact column-wise update $u_i^{k+1} = \argmin_{u \ge 0} h(u_{<i}^{k+1}, u, u_{>i}^k)$. Notably, the inclusion of the $\frac{\mu}{2}\|U\|_F^2$ term allows us to dispense with the standard assumption that $\|v_i\|_2$ is strictly bounded away from zero \cite[Assumption 1]{hou2023convergence}. Moreover, we establish a tighter, rank-independent subgradient bound compared to \cite[Proposition 1 (ii)]{hou2023convergence}.

\begin{lemma}[Sufficient Descent of HALS] \label{lem:descent}
Let $\{U^k\}$ be the sequence generated by the single-sided HALS. There exists $c_1 \ge \frac{\mu}{2} > 0$ such that $\mathcal{L}(U^k) - \mathcal{L}(U^{k+1}) \ge c_1 \|U^k - U^{k+1}\|_F^2$.
\end{lemma}

\begin{lemma}[Subgradient Bound of HALS] \label{lem:subgradient}
There exists a subgradient $W^{k+1} \in \partial \mathcal{L}(U^{k+1})$ satisfying $\|W^{k+1}\|_F \le c_2 \|U^k - U^{k+1}\|_F$, where $c_2 = \|M_L\|_2$ and $M_L$ is the strictly lower triangular part of $V^\top V$.
\end{lemma}
Proofs of \Cref{lem:descent,lem:subgradient} are given in Appendix~\ref{app:proof_hals_bounds}.

Building upon these local properties of the active-set subproblems, we now extend the analysis to the entire PASM sequence, establishing the sufficient decrease and relative error conditions required by the KL framework.

\begin{lemma}[Sufficient Decrease of PASM] \label{lem:sufficient_decrease}
Let $\{(U^k, V^k)\}$ be the sequence generated by Algorithm \ref{alg:ActiveSet}. There exists a constant $\rho = \min(\alpha_1, \mu) > 0$ such that for all $k \ge 0$:
{\small
\begin{multline}
\Psi(U^{k+1}, V^{k+1}) \le \Psi(U^k, V^k) \\- \frac{\rho}{2} \left( \|\hat{U}^{k+1} - \tilde{U}^{k+1}\|_F^2+\|\tilde{U}^{k+1} - U^k\|_F^2\right) \\- \frac{\rho}{2}\left(\|\hat{V}^{k+1} - \tilde{V}^{k+1}\|_F^2+\|\tilde{V}^{k+1} - V^k\|_F^2 \right).
\end{multline}
}
\end{lemma}

\begin{lemma}[Subgradient Bound of PASM] \label{lem:subgradient_bound}
Let $\{(U^k, V^k)\}$ be the sequence generated by Algorithm \ref{alg:ActiveSet}. There exist a constant $C=C_V+C_U+\sqrt{2}\mu+\tau+(\sqrt{2}+1)c_f$ and a subgradient $W^{k+1} \in \partial \Psi(U^{k+1}, V^{k+1})$ such that:
{\small
\begin{multline}
\|W^{k+1}\|_F \le C \left( \|\hat{U}^{k+1} - \tilde{U}^{k+1}\|_F+\|\tilde{U}^{k+1} - U^k\|_F \right. \\
\left. +\|\hat{V}^{k+1} - \tilde{V}^{k+1}\|_F+\|\tilde{V}^{k+1} - V^k\|_F \right).
\end{multline}
}
\end{lemma}
Proofs of \Cref{lem:sufficient_decrease,lem:subgradient_bound} are given in Appendix~\ref{app:proof_pasm_bounds}.

The decrease inequality, coercivity, and the proximal optimality conditions also imply that $\Psi$ is finite and constant on the compact cluster set of the PASM sequence, by the same lower-semicontinuity argument used for \Cref{prop:KL_properties-iPALM}.

Based on \Cref{thm:unified_global_convergence} with $Z^k = (U^k,V^k)$, $X^k = Z^k$, $\psi(X^k) = \Psi(U^k,V^k)$, $\Delta Z^k = Z^{k+1} - Z^k$, and $\mathcal{R}^k = \|\hat{U}^{k+1} - \tilde{U}^{k+1}\|_F+\|\tilde{U}^{k+1} - U^k\|_F+\|\hat{V}^{k+1} - \tilde{V}^{k+1}\|_F+\|\tilde{V}^{k+1} - V^k\|_F$, the sequence $\{Z^k\}_{k\in\mathbb{N}}$ globally converges to a critical point $Z^*$.

\section{Numerical Experiments}\label{sec:experiments}

In this section, we evaluate the efficiency and effectiveness of the three proposed algorithms: iPALM, Scale-Balanced iPALM (S-iPALM), and the Proximal Active-Set Method (PASM). The experiments comprise (1) an algorithmic efficiency comparison with the conventional PALM method; (2) diagnostic comparisons of the continuous $\lambda$-path with discrete $r$-paths and of S-iPALM with its unbalanced counterpart; (3) model-order recovery on synthetic data against Singular Value Hard Thresholding (SVHT) and Cross-Validation (CV); and (4) validation on diverse signal benchmarks spanning image, hyperspectral, single-cell, and document--term data. All experiments are implemented in MATLAB R2025b on a PC equipped with an Intel Core i9-14900HX 2.20GHz processor and 32GB of RAM. The MATLAB implementation and scripts used in the numerical experiments are publicly available at \url{https://github.com/Kemoralocy/ORS-L20-NMF}.

\subsection{Implementation Details}\label{sec:implementation}

Our stopping criterion first requires the stabilization of the support set, meaning the positive rank must remain constant for $T_1 = 20$ successive iterations:
\begin{equation}
    \mathrm{rank}_+(X^k) = \dots = \mathrm{rank}_+(X^{k-T_1}),
\end{equation}
where $\mathrm{rank}_+(X^j) = \min(\mathrm{rank}(U^j), \mathrm{rank}(V^j))$. Once stabilized, the algorithm terminates if any of the following conditions is met:

\begin{itemize}
    \item[\textbf{1.}] \textbf{Relative decrease of the objective function:}
    \begin{equation}
        \frac{|\Phi_{\lambda,\mu}(U^k,V^k) - \Phi_{\lambda,\mu}(U^{k-T_2},V^{k-T_2})|}{\max(1, |\Phi_{\lambda,\mu}(U^k,V^k)|)} \le \epsilon.
    \end{equation}
    
    \item[\textbf{2.}] \textbf{Relative change of the variables:}
    \begin{equation}
        \frac{\|(U^k,V^k) - (U^{k-T_3},V^{k-T_3})\|_F}{\max(1, \|(U^k,V^k)\|_F)} \le \epsilon.
    \end{equation}
    
    \item[\textbf{3.}] \textbf{Projected gradient stationarity:}
    Given that Algorithm 3 is a two-stage method and structurally distinct from Algorithms 1 and 2, we define the projected gradient for $U$ as follows:
    {\small
    \begin{multline}
        (\nabla^p_U F_{\mu}(U,V))_{ij} \\= 
        \begin{cases} 
            (\nabla_U F_{\mu}(U,V))_{ij}, &\begin{multlined}[t]
             \text{if } (\nabla_U F_{\mu}(U,V))_{ij} < 0 \text{ or } U_{ij} > 0,    
            \end{multlined} \\
            0, & \text{otherwise}.
        \end{cases}
    \end{multline}
    }%
The projected gradient $\nabla^p_V F_{\mu}(U,V)$ is defined analogously. According to \Cref{prop:partial}, it directly follows that $\nabla^p_U F_{\mu}(U^k,V^k) \in \partial_U \Psi_{\lambda,\mu}(U^k,V^k)$ and $\nabla^p_V F_{\mu}(U^k,V^k) \in \partial_V \Psi_{\lambda,\mu}(U^k,V^k)$. Thus, the termination condition is given by:
    \begin{equation}
        \|(\nabla^p_U F_{\mu}(U^k,V^k), \nabla^p_V F_{\mu}(U^k,V^k))\|_F \le \epsilon_1.
    \end{equation}
    However, it is worth noting that all gradient-based stopping criteria remain sensitive to scaling \cite{kim2011fast,kim2014algorithms,gillis2020nonnegative}.
\end{itemize}

We empirically set $T_2 = T_3 = 10$, $\epsilon = 10^{-4}$, and $\epsilon_1 = 10^{-3}$. For the proposed algorithms, the step sizes are selected as $\gamma_{1,k}=\tau_{\!V^k}+\delta$ and $\gamma_{2,k}= \tau_{\!U^{k+1}} + \delta$ with $\delta=10^{-6}$. We employ Nesterov's accelerated strategy
  \cite{Nesterov83} to yield $\beta_k$ of Algorithm \ref{alg:iPALM} and \ref{alg:iPALM_scale}, i.e.,
  $\beta_k=\frac{t_{k-1}-1}{t_{k}}$ with $t_{-1}=t_0=1$ and $t_{k+1}=\frac{1+\sqrt{4t_k^2+1}}{2}$.
  Though our convergence results require a restriction on $\beta_k$, numerical tests
  show that Algorithm \ref{alg:iPALM} and \ref{alg:iPALM_scale} still converge without it. In view of this,
  we do not impose any restriction on such $\beta_k$ for the subsequent tests,
  and leave this gap for a future research topic. Across all experiments, we initialize the estimated rank to $r=\min(\lceil0.5\min(n,m)\rceil,50)$ and fix $\mu=10^{-8}$.

Next, we specify the adaptive choice of $\lambda$. For the $i$th warm start, let $\hat{P}^{i,0}$ denote $\hat{P}^0$ in \Cref{remark1-iPALM}, write $\gamma_i$ for its $U$-block step size, and let $g^{(i)\downarrow}$ be the nonincreasing rearrangement of $\{\|(\hat{P}^{i,0})_j\|_2^2:j\in\mathcal{T}_i\}$, where $\mathcal{T}_i$ is the active support and $q_i=|\mathcal{T}_i|$. The rule therein shows that column $j$ is retained precisely when $\lambda<\frac{1}{2}\gamma_i\|(\hat{P}^{i,0})_j\|_2^2$. Therefore, with $\varsigma=10^{-8}$, we set
\begin{equation}\label{eq:adaptive_lambda_path}
\begin{aligned}
\lambda_1 = \tfrac{1}{2}(1+\varsigma)\gamma_0
g_{q_0}^{(0)\downarrow}&, \quad
\overline{\lambda}_i = \tfrac{1}{2}(1+\varsigma)\gamma_i
g_2^{(i)\downarrow},\\
\lambda_{i+1} =& \lambda_i+
\frac{\overline{\lambda}_i-\lambda_i}{N_\lambda-i},
\end{aligned}
\end{equation}
where $\lambda_1$ crosses the smallest initial breakpoint and $\overline{\lambda}_i$ crosses the second-largest current breakpoint. Since $U$ is updated first, the $U$-side quantities are used; for PASM, the same construction uses $Z_U^0$ in place of $\hat P^0$. Each subproblem is warm-started from the preceding solution. Unless otherwise specified, the path terminates at rank one or after $N_\lambda=30$ points; the diagnostic experiments in \Cref{sec:lambda_path_analysis,sec:scale_balancing_experiment} use $N_\lambda=20$. From the solutions associated with distinct ranks, we compute
\begin{equation}\label{eq:marginal_rank_score}
\vartheta(i)=
\frac{|\mathrm{loss}(i-1)-\mathrm{loss}(i)|}
{|\mathrm{rank}(i-1)-\mathrm{rank}(i)|}.
\end{equation}
With $\mathrm{loss}(0)=0.5(\|X\|_F^2-\sigma_1(X)^2)$ and $\mathrm{rank}(0)=1$, we select $\lambda_{i^*-1}$, where $i^*$ is the first index satisfying $\vartheta(i^*-1)/\vartheta(i^*)>\tau$; we use $\tau=5$ for synthetic data and $\tau=2$ for real-world data.

\subsection{Algorithmic Efficiency Comparison}\label{sec:efficiency}

We first evaluate the computational efficiency of PALM, iPALM, S-iPALM, and PASM in solving model \cref{eq:l20_regularized_nmf} on small-scale ($n = 100, m = 80$) and large-scale ($n = 1000, m = 800$) synthetic non-negative matrices. All four algorithms are strictly initialized with the exact same random non-negative starting point $(U^0, V^0)$, an overestimated rank $r = 30$, and fixed penalty parameters $\lambda$ and $\mu$.

\begin{figure}
    \centering
    \includegraphics[width=1\linewidth]{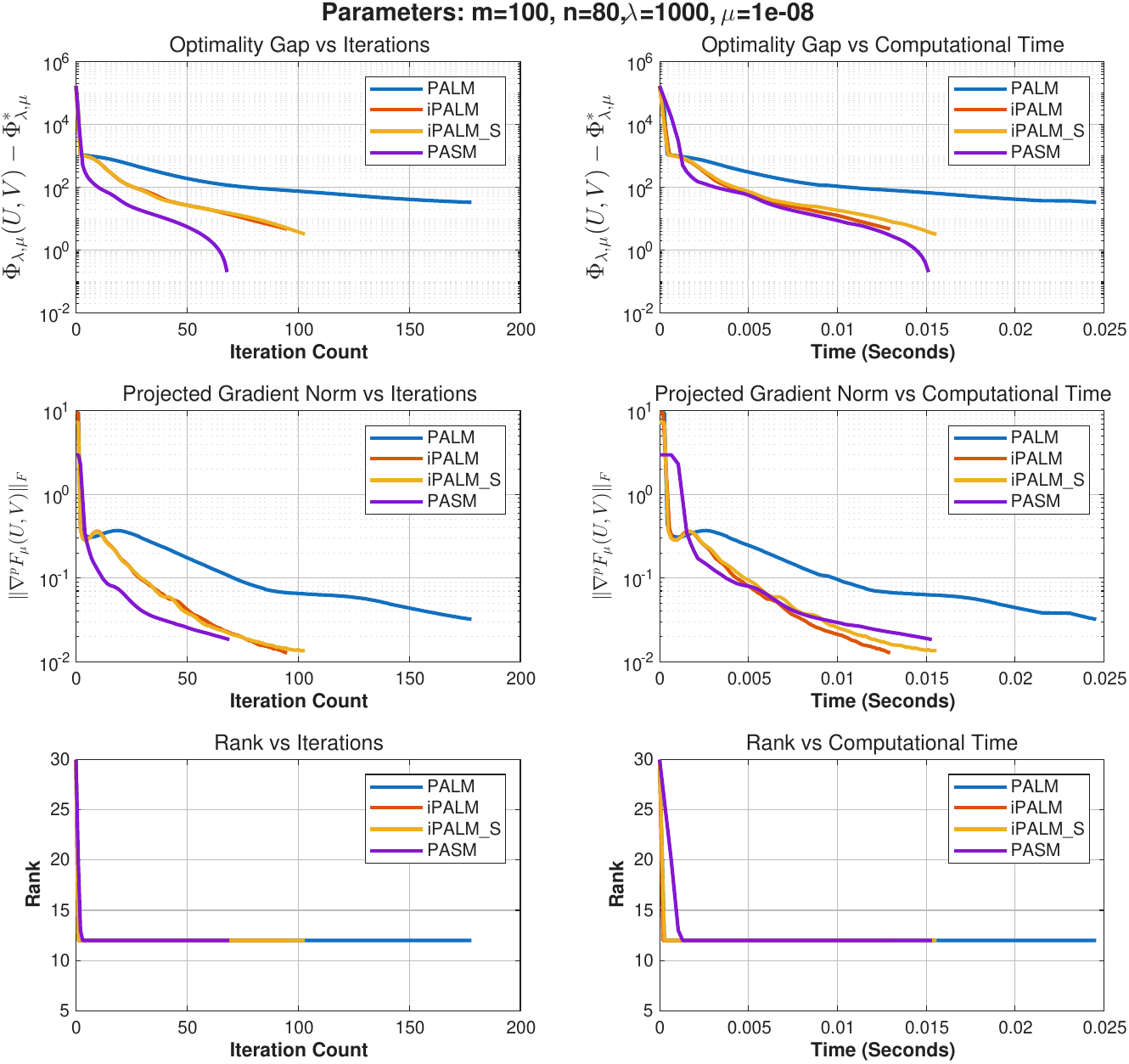}
    \includegraphics[width=1\linewidth]{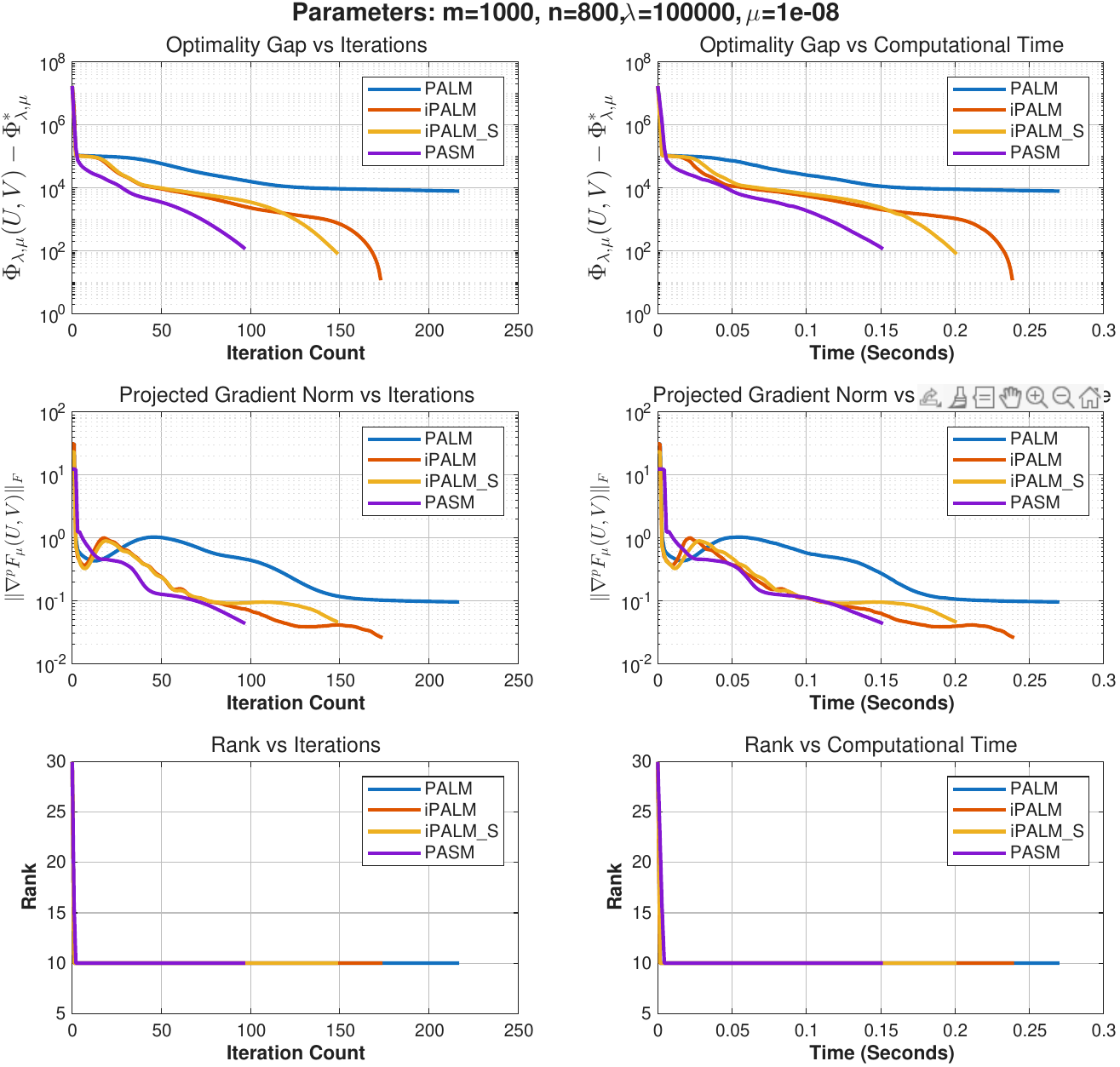}
    \caption{Comparison of computational efficiency among PALM, iPALM, S-iPALM, and PASM: objective function value, projected gradient norm and rank versus iterations and CPU running time.}
    \label{fig:efficiency}
\end{figure}

Figure \ref{fig:efficiency} plots the objective function values against the iteration count and computational time. Both iPALM and S-iPALM exhibit significantly faster objective descent per iteration compared to the baseline PALM. Most notably, the proposed PASM demonstrates an overwhelming advantage in both iteration efficiency and overall running time.

\subsection{Continuous \texorpdfstring{$\lambda$}{lambda}-Path versus Discrete Rank Search}\label{sec:lambda_path_analysis}

We next illustrate how the warm-started $\lambda$-path eliminates redundant components and why it is preferable to a discrete $r$-path. The diagnostic experiments use synthetic data generated as described in \Cref{sec:synthetic} with $m=100$, $n=80$, and $r^*=10$, together with the Swimmer dataset in \Cref{sec:real_data}, whose reference rank is $r^*=16$. We set $\mu=10^{-8}$ and $N_\lambda=20$ and use the stopping tolerances specified in \Cref{sec:implementation}.

\begin{figure}[t]
    \centering
    \includegraphics[width=\linewidth]{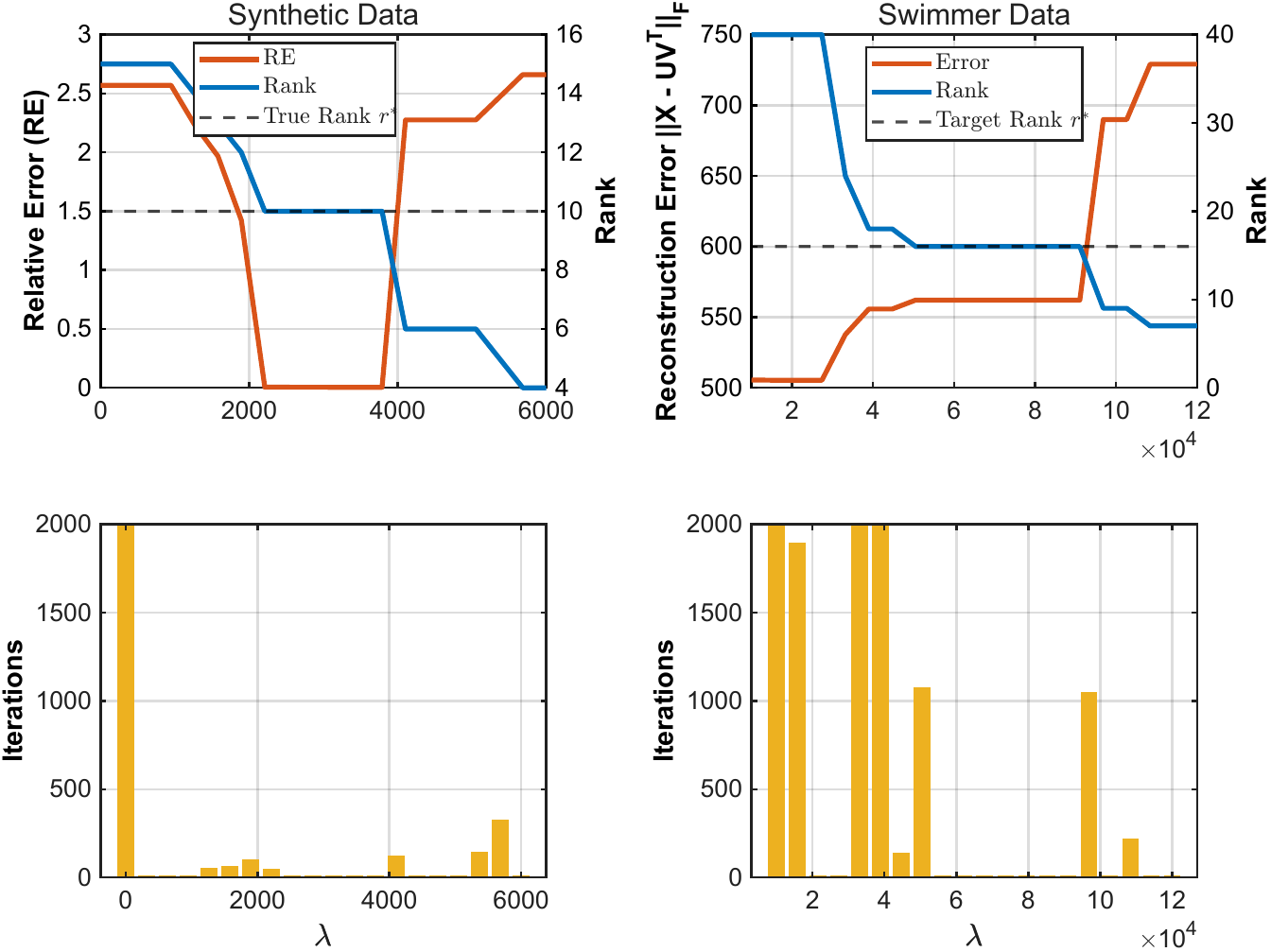}
    \caption{Continuous $\lambda$-paths and warm-start iteration counts on the synthetic and Swimmer datasets. The settings are $m=100$, $n=80$, and $r^*=10$ for the synthetic data; $r^*=16$ for Swimmer; $\mu=10^{-8}$; and $N_\lambda=20$.}
    \label{fig:lambda_path}
\end{figure}

As shown in \Cref{fig:lambda_path}, the estimated rank changes in a piecewise-constant manner and exhibits a clear plateau at the reference rank, where the reconstruction error is also low. Apart from the initial solve and several rank-transition points, the warm-started subproblems require relatively few iterations, indicating that solutions can be efficiently reused within a rank plateau.

\begin{figure}[t]
    \centering
    \includegraphics[width=\linewidth]{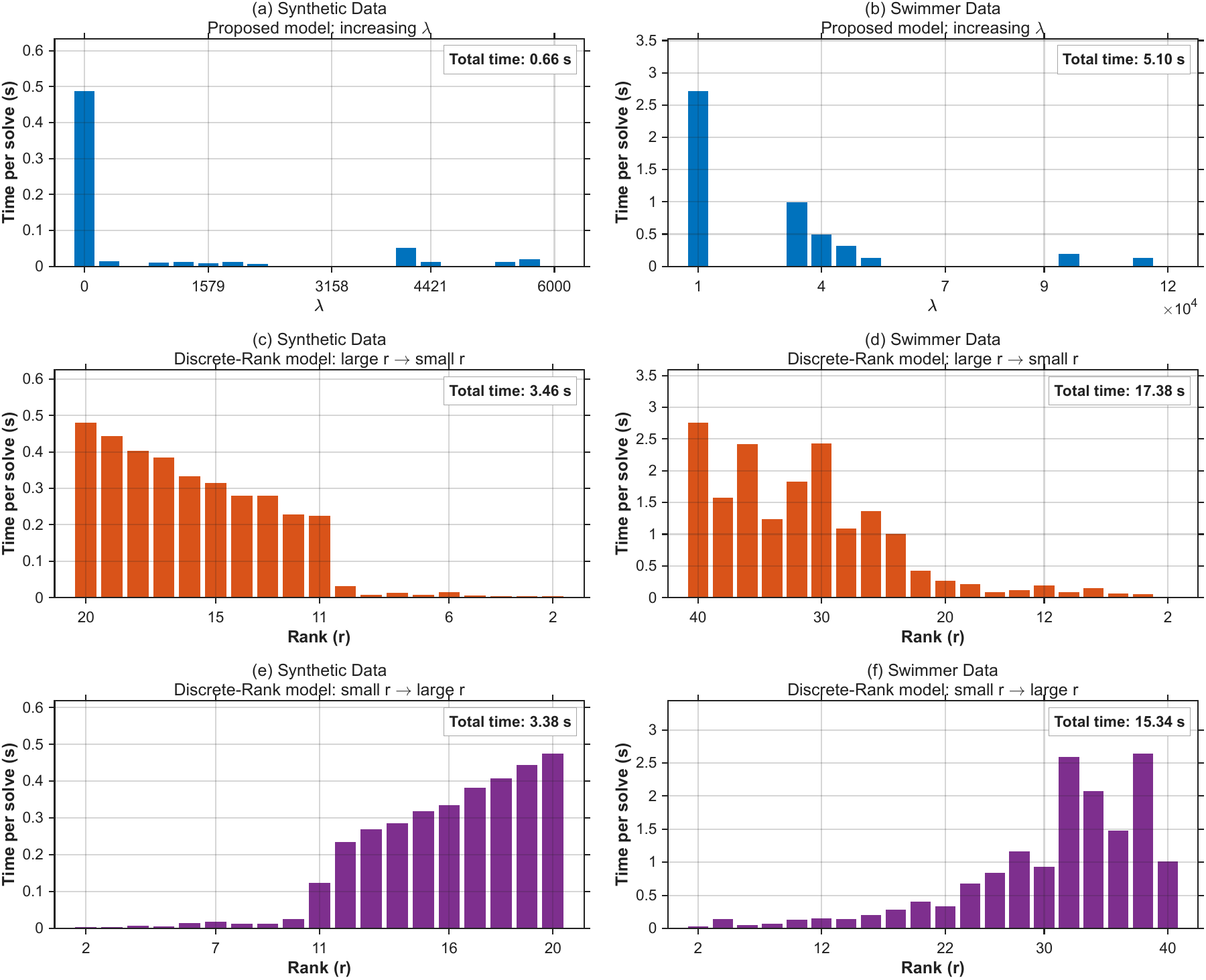}
    \caption{Computational-time comparison between the continuous $\lambda$-path and two discrete $r$-paths of the Discrete-Rank model (i.e., $\lambda=0$ in \eqref{eq:l20_regularized_nmf}, with $r$ varied from large to small or from small to large). The discrete paths use $r\in\{2,\ldots,20\}$ for the synthetic data and $r\in\{2,\ldots,40\}$ for Swimmer; $\mu=10^{-8}$ and $N_\lambda=20$.}
    \label{fig:time_comparison}
\end{figure}

\Cref{fig:time_comparison} further compares the continuous $\lambda$-path with two discrete $r$-paths. On the synthetic data, the proposed path requires $0.66$ seconds in total, compared with $3.46$ and $3.38$ seconds for the decreasing- and increasing-rank paths, respectively. On Swimmer, the corresponding times are $5.10$, $17.38$, and $15.34$ seconds. A discrete $r$-path must repeatedly change the dimensions of the factor matrices: decreasing $r$ discards learned components, whereas increasing $r$ introduces randomly initialized columns. Both operations disrupt the current solution and require substantial re-optimization at successive ranks. In contrast, the continuous $\lambda$-path retains fixed factor dimensions and progressively removes redundant components, allowing the information accumulated at preceding path points to be preserved.

\subsection{Effect of Scale Balancing on Rank Selection}\label{sec:scale_balancing_experiment}

Using the same datasets and parameter settings as in \Cref{sec:lambda_path_analysis}, we compare S-iPALM with the original iPALM to evaluate the effect of scale balancing on the $\lambda$-path.

\begin{figure}[t]
    \centering
    \includegraphics[width=\linewidth]{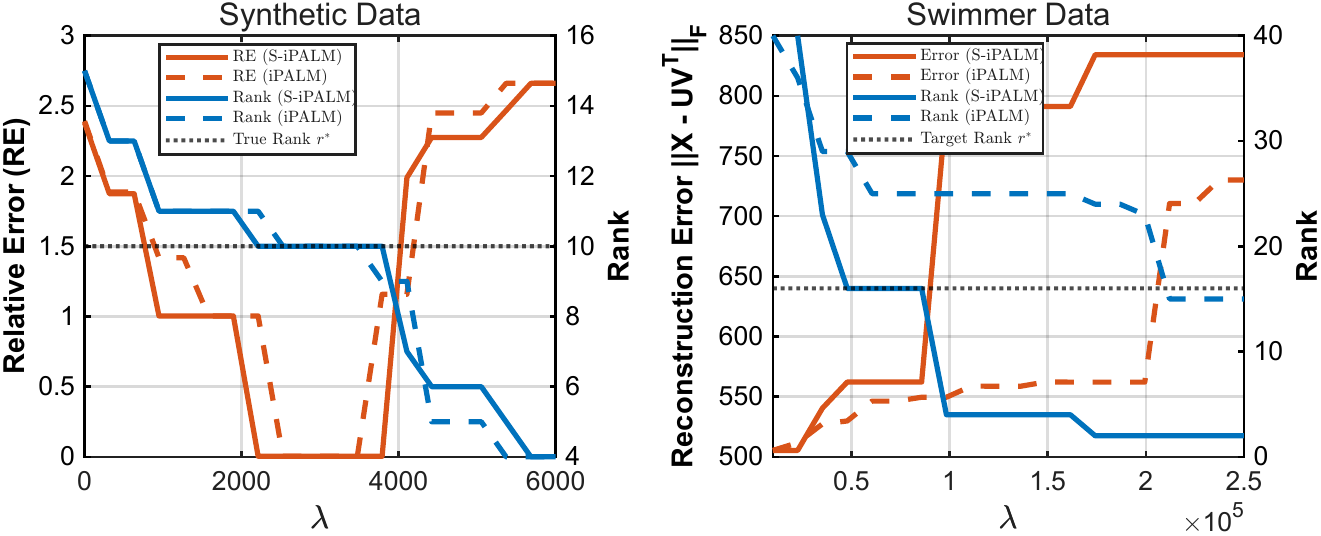}
    \caption{Effect of scale balancing on the $\lambda$-paths for the synthetic and Swimmer datasets. S-iPALM denotes the scale-balanced method and iPALM its unbalanced counterpart; $\mu=10^{-8}$ and $N_\lambda=20$.}
    \label{fig:scale_comparison}
\end{figure}

\Cref{fig:scale_comparison} shows that scale balancing improves the stability of the rank-selection path. For the synthetic data, S-iPALM produces a wider plateau at the true rank and maintains near-zero recovery error over that plateau. For Swimmer, S-iPALM reaches the reference rank at a substantially smaller value of $\lambda$ and with a lower reconstruction error, whereas the unbalanced method remains over-ranked over much of the path.

In view of the algorithmic efficiency results in \Cref{sec:efficiency} and the scale-balancing behavior observed above, the subsequent synthetic experiments use iPALM without scale balancing, whereas the real-world experiments use its scale-balanced variant; for brevity, both are reported as iPALM.

\subsection{Numerical Results for Synthetic Data}\label{sec:synthetic}

Next, we compare the rank recovery performance of iPALM and PASM against SVHT and CV on synthetic data, as a recent comprehensive evaluation \cite{eswar2024rank} has identified them as the most promising and reliable practical baselines for NMF rank selection. Performance is evaluated using the recovery error (RE) between the ground-truth factor $W^*$ and the estimated factor $W$. To remove the scaling ambiguity of NMF \cite{gillis2020nonnegative,landy2026bayesnmf}, all nonzero columns are normalized to unit $\ell_2$-norm, yielding $\bar W^*$ and $\bar W$, while zero columns in $W$ remain zero. Following standard factor-alignment practice \cite{qu2015subspace,kuhn1955hungarian}, let $\Pi^*$ denote the optimal partial permutation obtained by the Hungarian algorithm, and let $\mathcal M$ be the indices of the matched columns in $\bar W$. We define
\begin{equation}
    \operatorname{RE}(W^*,W)
    =
    \left\|\bar W^*-\bar W\Pi^*\right\|_F
    +
    \left\|\bar W_{:,\mathcal M^c}\right\|_F .
\end{equation}
The first term measures the optimally aligned factor error and automatically penalizes rank underestimation by matching missing components with zero columns, whereas the second term penalizes unmatched columns caused by rank overestimation. Thus, the metric accounts for the scaling and permutation ambiguities of NMF as well as rank mismatch.

The synthetic observation matrix is generated as $X = \max(0, U^*(V^*)^\top + \sigma \mathcal{E})$, where factors $U^* \in \mathbb{R}^{n \times r^*}$ and $V^* \in \mathbb{R}^{m \times r^*}$ are independently sampled from $\mathcal{N}(0,1)$ and truncated to non-negativity. We evaluate matrix dimensions from $(1000, 400)$ to $(5000, 4000)$, varying the true rank $r^* \in \{10, 20\}$ and noise level $\sigma \in \{0.1, 0.3\}$. Both iPALM and PASM are capped at 2000 iterations with a tolerance of $10^{-4}$.

\begin{table*}[htbp]
\centering
\caption{Average RE, predicted rank, and running time (in seconds) of four methods for synthetic datasets. `-' indicates that the relative error is not applicable or the algorithm failed to complete within acceptable time.}
\label{tab:synthetic}
\scalebox{0.8}{
\begin{tabular}{cccc|ccc|ccc|ccc|ccc}
\toprule
\multirow{2}{*}{$n$} & \multirow{2}{*}{$m$} & \multirow{2}{*}{$\sigma$} & \multirow{2}{*}{$r^*$} & \multicolumn{3}{c}{iPALM} & \multicolumn{3}{c}{PASM} & \multicolumn{3}{c}{SVHT} & \multicolumn{3}{c}{CV} \\
\cmidrule(lr){5-7} \cmidrule(lr){8-10} \cmidrule(lr){11-13} \cmidrule(lr){14-16}
 & & & & Rank & RE & Time & Rank & RE & Time & Rank & RE & Time & Rank & RE & Time \\
\midrule
\begin{filecontents*}{NMF_Synthetic_Experiments_Formatted_m1000.csv}
m,n,sigma,r_star,iPALM_best_rank,iPALM_RE,iPALM_time,PASM_best_rank,PASM_RE,PASM_time,SVHT_best_rank,SVHT_RE,SVHT_time,CV_best_rank,CV_RE,CV_time
1000,400,0.1,10,10,3.459015e-02,0.58,10,3.459015e-02,0.65,10,-,0.02,10,3.459015e-02,187.57
1000,400,0.1,20,20,4.677367e-02,0.71,20,4.673877e-02,0.81,20,-,0.01,20,4.675550e-02,149.67
1000,400,0.3,10,10,1.070742e-01,0.50,10,1.070742e-01,0.75,10,-,0.01,10,1.070495e-01,152.63
1000,400,0.3,20,20,1.409395e-01,0.65,20,1.409342e-01,0.84,20,-,0.01,20,1.409780e-01,131.46
1000,800,0.1,10,10,2.507372e-02,0.76,10,2.507372e-02,0.98,10,-,0.04,10,2.507298e-02,386.13
1000,800,0.1,20,20,3.211601e-02,1.02,20,3.210953e-02,1.19,20,-,0.04,20,3.211727e-02,301.11
1000,800,0.3,10,10,7.918155e-02,0.90,10,7.918155e-02,1.33,10,-,0.06,10,7.916743e-02,322.81
1000,800,0.3,20,20,9.723429e-02,1.43,20,9.723652e-02,1.36,20,-,0.06,20,9.724437e-02,287.89
\end{filecontents*}
\csvreader[
    late after line=\\, 
]{NMF_Synthetic_Experiments_Formatted_m1000.csv}{
    m=\m, n=\n, sigma=\sig, r_star=\rstar,
    iPALM_best_rank=\iprank, iPALM_RE=\ipre, iPALM_time=\iptime,
    PASM_best_rank=\pasrank, PASM_RE=\pasre, PASM_time=\pastime,
    SVHT_best_rank=\svhtrank, SVHT_RE=\svhtre, SVHT_time=\svhttime,
    CV_best_rank=\cvrank, CV_RE=\cvre, CV_time=\cvtime
}{%
    \m & \n & \sig & \rstar & 
    \iprank & \ipre & \iptime & 
    \pasrank & \pasre & \pastime & 
    \svhtrank & \svhtre & \svhttime & 
    \cvrank & \cvre & \cvtime
}
\midrule
\begin{filecontents*}{NMF_Synthetic_Experiments_Formatted_m3000.csv}
m,n,sigma,r_star,iPALM_best_rank,iPALM_RE,iPALM_time,PASM_best_rank,PASM_RE,PASM_time,SVHT_best_rank,SVHT_RE,SVHT_time,CV_best_rank,CV_RE,CV_time
3000,1000,0.1,10,10,2.311187e-02,2.42,10,2.311187e-02,3.54,10,-,0.09,10,2.310838e-02,917.15
3000,1000,0.1,20,20,2.847521e-02,4.09,20,2.842076e-02,3.99,20,-,0.11,20,2.843069e-02,946.61
3000,1000,0.3,10,10,7.348318e-02,3.02,10,7.348318e-02,4.43,11,-,0.13,10,7.346473e-02,906.29
3000,1000,0.3,20,22,1.965367e+00,4.10,20,8.626572e-02,3.83,20,-,0.11,20,8.627458e-02,937.76
3000,2000,0.1,10,10,1.795882e-02,5.75,10,1.795882e-02,5.97,11,-,0.55,10,1.795499e-02,1499.15
3000,2000,0.1,20,20,1.967294e-02,6.86,20,1.967327e-02,6.76,20,-,0.50,20,1.968731e-02,1678.68
3000,2000,0.3,10,10,5.864614e-02,5.62,10,5.864615e-02,5.87,11,-,0.51,10,5.863489e-02,1564.50
3000,2000,0.3,20,20,6.007926e-02,6.36,20,6.007905e-02,6.97,20,-,0.46,20,6.008414e-02,1704.50
\end{filecontents*}
\csvreader[
    late after line=\\, 
]{NMF_Synthetic_Experiments_Formatted_m3000.csv}{
    m=\m, n=\n, sigma=\sig, r_star=\rstar,
    iPALM_best_rank=\iprank, iPALM_RE=\ipre, iPALM_time=\iptime,
    PASM_best_rank=\pasrank, PASM_RE=\pasre, PASM_time=\pastime,
    SVHT_best_rank=\svhtrank, SVHT_RE=\svhtre, SVHT_time=\svhttime,
    CV_best_rank=\cvrank, CV_RE=\cvre, CV_time=\cvtime
}{%
    \m & \n & \sig & \rstar & 
    \iprank & \ipre & \iptime & 
    \pasrank & \pasre & \pastime & 
    \svhtrank & \svhtre & \svhttime & 
    \cvrank & \cvre & \cvtime
}
\midrule
\begin{filecontents*}{NMF_Synthetic_Experiments_Formatted_m5000.csv}
m,n,sigma,r_star,iPALM_best_rank,iPALM_RE,iPALM_time,PASM_best_rank,PASM_RE,PASM_time,SVHT_best_rank,SVHT_RE,SVHT_time,CV_best_rank,CV_RE,CV_time
5000,3000,0.1,10,10,1.530275e-02,8.79,10,1.530275e-02,9.43,11,-,1.28,-,-,-
5000,3000,0.1,20,20,1.604520e-02,8.86,20,1.604530e-02,10.01,20,-,1.28,-,-,-
5000,3000,0.3,10,10,5.100082e-02,7.74,10,5.100082e-02,9.35,11,-,1.26,-,-,-
5000,3000,0.3,20,20,4.929761e-02,9.76,20,4.929794e-02,9.82,20,-,1.26,-,-,-
5000,4000,0.1,10,10,1.402154e-02,10.47,10,1.402154e-02,12.14,11,-,2.11,-,-,-
5000,4000,0.1,20,20,1.394394e-02,11.85,20,1.394398e-02,12.66,20,-,2.07,-,-,-
5000,4000,0.3,10,10,4.734162e-02,10.93,10,4.734162e-02,12.57,11,-,2.12,-,-,-
5000,4000,0.3,20,20,4.306148e-02,12.28,20,4.306081e-02,12.64,20,-,2.08,-,-,-
\end{filecontents*}
\csvreader[
    late after line=\\, 
]{NMF_Synthetic_Experiments_Formatted_m5000.csv}{
    m=\m, n=\n, sigma=\sig, r_star=\rstar,
    iPALM_best_rank=\iprank, iPALM_RE=\ipre, iPALM_time=\iptime,
    PASM_best_rank=\pasrank, PASM_RE=\pasre, PASM_time=\pastime,
    SVHT_best_rank=\svhtrank, SVHT_RE=\svhtre, SVHT_time=\svhttime,
    CV_best_rank=\cvrank, CV_RE=\cvre, CV_time=\cvtime
}{%
    \m & \n & \sig & \rstar & 
    \iprank & \ipre & \iptime & 
    \pasrank & \pasre & \pastime & 
    \svhtrank & \svhtre & \svhttime & 
    \cvrank & \cvre & \cvtime
}
\bottomrule
\end{tabular}}
\end{table*}

As detailed in Table \ref{tab:synthetic}, PASM perfectly recovers the true rank $r^*$ across all dimensions and noise levels with 100\% accuracy. iPALM is comparably precise, overestimating by only a single rank under strong noise conditions ($n=3000, m=1000, \sigma=0.3$). SVHT, while computationally fast as it avoids full factorizations, suffers from systematic rank overestimation as dimensions scale up. When accurately predicting the rank, both iPALM and PASM achieve exceptionally low relative errors ($10^{-2}$ to $10^{-1}$), directly matching the benchmark performance of CV. Furthermore, utilizing the $\lambda$-path warm-start grants iPALM and PASM excellent scalability, converging in seconds, whereas the exhaustive CV method faces severe computational bottlenecks, entirely failing on $n=5000$ matrices.

\subsection{Performance on Diverse Signal Benchmarks}\label{sec:real_data}

We finally evaluate model-order selection on seven benchmark matrices spanning single-cell expression, image decomposition, document--term representations, and hyperspectral unmixing. The reported number of cell types, image components or subjects, document categories, or spectral endmembers is used as the reference order; as noted in the Introduction, small discrepancies from the exact nonnegative rank may arise from noise and model mismatch.

\begin{itemize}
    \item \textbf{PBMC3K \cite{zheng2017massively}:} The $2700\times13714$ single-cell RNA-sequencing matrix contains nine annotated cell types, giving a reference order of 9.
    \item \textbf{Swimmer\cite{donoho2003does}:} The $256 \times 1024$ nonnegative image matrix contains 1024 stick-figure images generated from 16 independent limb positions, giving a reference order of 16.
    \item \textbf{AT\&T Faces \cite{samaria1994parameterisation}:} The $400\times4096$ image matrix contains ten subjects, giving a reference order of 10.
    \item \textbf{Web of Science (WoS) \cite{kowsari2017HDLTex}:} The $5736\times20082$ document--term matrix contains eleven subcategories, giving a reference order of 11.
    \item \textbf{Samson \cite{zhu2017hyperspectral}:} The hyperspectral matrix contains 156 spectral bands and $95\times95=9025$ pixels. Its three reported endmembers give a reference order of 3.
    \item \textbf{Jasper Ridge \cite{zhu2017hyperspectral}:} The hyperspectral matrix contains 224 spectral bands and $100\times100=10000$ pixels. Its four reported endmembers give a reference order of 4.
    \item \textbf{Urban \cite{zhu2017hyperspectral}:} After removing noisy and water-absorption bands, the hyperspectral matrix contains 162 bands and $307\times307=94249$ pixels. We use the four-endmember reference configuration.
\end{itemize}

Because ground-truth factors are unavailable for these real datasets, we obtain a surrogate reference factor $W_{\mathrm{ref}}$ by running NMF at the adopted reference order $r_{\mathrm{ref}}$. We compute the REs of iPALM and PASM relative to $W_{\mathrm{ref}}$ using the same column normalization and Hungarian alignment as above; for SVHT, we additionally run NMF at the order selected by SVHT and evaluate the resulting factor in the same way. CV is omitted because its repeated factorizations over candidate orders require prohibitively long runtimes on these datasets.

\begin{table*}[t]
\centering
\caption{Estimated model orders, recovery errors (REs),
and execution times (seconds) on the diverse signal benchmarks.}
\label{tab:performance_comparison}

\scriptsize
\setlength{\tabcolsep}{2.2pt}
\renewcommand{\arraystretch}{1.08}

\begin{tabular}{lcc ccc ccc ccc}
\toprule
\multirow{2}{*}{\textbf{Dataset}}
& \multirow{2}{*}{\textbf{Dimension}}
& \multirow{2}{*}{\textbf{Ref. order}}
& \multicolumn{3}{c}{\textbf{iPALM}}
& \multicolumn{3}{c}{\textbf{PASM}}
& \multicolumn{3}{c}{\textbf{SVHT}} \\
\cmidrule(lr){4-6}
\cmidrule(lr){7-9}
\cmidrule(lr){10-12}
& & 
& \textbf{Rank} & $\mathbf{RE}$ & \textbf{Time (s)}
& \textbf{Rank} & $\mathbf{RE}$ & \textbf{Time (s)}
& \textbf{Rank} & $\mathbf{RE}$ & \textbf{Time (s)} \\
\midrule
PBMC3K
& $2700 \times 13714$
& 9
& 9   & 0.9037 & 25.15
& 9   & 0.7312 & 24.11
& 665 & 27.025 & 8.05 \\

Swimmer
& $256 \times 1024$
& 16
& 16 & 1.5423 & 0.89
& 16 & 1.9717 & 0.72
& 22 & 4.670 & 0.08 \\

AT\&T Faces
& $400 \times 4096$
& 10
& 10  & 6.0507 & 1.75
& 10  & 2.5812 & 1.13
& 114 & 12.412 & 0.12 \\

Web of Science
& $5736 \times 20082$
& 11
& 10   & 1.9683 & 9.93
& 12   & 2.5676 & 25.68
& 1162 & 36.597 & 9.04 \\

Samson
& $156 \times 9025$
& 3
& 5  & 1.7656 & 3.24
& 3  & 0.4957 & 4.30
& 61 & 8.335 & 0.06 \\

Jasper Ridge
& $224 \times 10000$
& 4
& 4  & 0.3885 & 4.83
& 4  & 0.6040 & 5.59
& 86 & 10.020 & 0.03 \\

Urban
& $162 \times 94249$
& 4
& 4  & 0.2777 & 25.54
& 4  & 0.4395 & 34.58
& 55 & 8.455 & 0.08 \\
\bottomrule
\end{tabular}
\end{table*}

Table~\ref{tab:performance_comparison} summarizes the results. PASM matches the adopted reference order on six of the seven benchmarks and overestimates Web of Science by one component, while iPALM matches five benchmarks, underestimates Web of Science by one component, and overestimates Samson by two. Both methods achieve substantially lower REs than SVHT on every dataset, whereas SVHT strongly overestimates all reference orders. PASM is faster on PBMC3K, Swimmer, and AT\&T Faces, while iPALM is faster on Web of Science and the three hyperspectral datasets. Overall, iPALM and PASM both provide accurate model-order estimates, low factor-recovery errors, and practical runtimes across these diverse signal benchmarks.

\section{Conclusions}
\label{sec:conclusions}

In this paper, we proposed a column $\ell_{2,0}$-norm regularized NMF model for estimating the factorization rank from a rough upper bound. A warm-started $\lambda$-path progressively eliminates redundant components without repeatedly changing the dimensions of the factor matrices. From a theoretical perspective, we established conditions under which critical points with suitably bounded objective values recover the underlying nonnegative rank, and characterized the relationships between the critical points and local minimizers of the proposed model and its smooth auxiliary counterpart. We also quantified the possible reconstruction-error floor induced by a positive squared Frobenius-norm regularization parameter, which motivates combining a sufficiently small $\mu$ with the column $\ell_{2,0}$ penalty. For computing the proposed model, we developed iPALM, its scale-balanced variant S-iPALM, and a Proximal Active-Set Method (PASM). The scale-balancing step preserves the reconstructed matrix while equalizing the norms of paired factor columns, and PASM exploits the identified active support to reduce the dimension of the NNLS subproblems. Under the stated algorithmic conditions, the Kurdyka--\L{}ojasiewicz analysis establishes whole-sequence convergence of the proposed methods to critical points. Numerical experiments show that both iPALM and PASM provide reliable order estimates across the synthetic configurations and seven diverse signal benchmarks while retaining practical runtimes. Together with the $\lambda$-path and scale-balancing ablations, the comparisons with SVHT and cross-validation demonstrate favorable order-selection accuracy and computational efficiency in the tested settings. Future work will investigate the statistical properties of the adaptive $\lambda$-path, relax the assumptions required for exact rank recovery, and extend the proposed framework to tensor and large-scale stochastic settings.

\appendices

\section{Proofs of the Structural Propositions}
\label{app:structural}

\begin{proof}[Proof of Proposition~\ref{prop:crit}]
Let $J_U$ and $J_V$ be the active column sets of $U$ and $V$. At a critical point of $\Psi_{\lambda,\mu}$, every $j\in J_U$ satisfies the nonnegative KKT conditions for the smooth part. If $V_j=0$, complementary slackness gives
\[
0=\left\langle (UV^\top-X)V_j+\mu U_j,U_j\right\rangle
=\mu\|U_j\|_2^2,
\]
contradicting $j\in J_U$. Hence $J_U\subseteq J_V$; symmetry gives $J_U=J_V=:J$. On $J$, the criticality systems of $\Psi_{\lambda,\mu}$ and $\widetilde F_\mu$ coincide, while on $J^c$ both smooth block gradients vanish because $U_j=V_j=0$. Thus $\crit\Psi_{\lambda,\mu}\subseteq\crit\widetilde F_\mu$. Conversely, the same support argument applies to a critical point of $\widetilde F_\mu$; on zero columns the limiting subdifferential of the column $\ell_{2,0}$ term is the whole ambient space, so the KKT system also gives criticality for $\Psi_{\lambda,\mu}$. Therefore the two critical sets agree.

Now let $(U^*,V^*)$ be a global minimizer, set $Y=U^*(V^*)^\top$, $q=|J|$, and $s=\rank_+(Y)$. The case $Y=0$ is immediate. Otherwise, if $q\ge s+1$, choose a balanced $s$-term nonnegative factorization $Y=AB^\top$ and zero-pad it to $r$ columns. By Theorem~\ref{thm:nmf_norm_bound}, with $p(A,B):=(\|A\|_F^2+\|B\|_F^2)/2$,
\[
p(A,B)\le \sqrt{s}\,\|Y\|_F.
\]
Since $p(U^*,V^*)>0$, the assumed bound on $\lambda/\mu$ yields
\begin{align*}
&\Psi_{\lambda,\mu}(A,B)-\Psi_{\lambda,\mu}(U^*,V^*)\\
&=\mu p(A,B)-\mu p(U^*,V^*)-2\lambda(q-s)\\
&<\mu\sqrt{s}\,\|Y\|_F-2\lambda(q-s)\le0,
\end{align*}
contradicting global optimality. Hence $q=s$, which proves the second assertion.
\end{proof}

\begin{proof}[Proof of Proposition~\ref{prop:lower bound}]
Put $Y=\overline U\overline V^\top$. Taking inner products of the two block KKT conditions with $\overline U$ and $\overline V$ gives
\[
\langle Y-X,Y\rangle+\mu\|\overline U\|_F^2=0,
\qquad
\langle Y-X,Y\rangle+\mu\|\overline V\|_F^2=0.
\]
Thus $\|\overline U\|_F=\|\overline V\|_F$ and, using $\|Y\|_F\le\|\overline U\|_F\|\overline V\|_F$,
\[
\mu\|Y\|_F
\le \mu\|\overline U\|_F^2
=\langle X-Y,Y\rangle
\le\|X-Y\|_F\|Y\|_F.
\]
Proposition~\ref{prop:crit} and nonnegativity imply $Y\ne0$ for a nonzero critical point, so $\|X-Y\|_F\ge\mu$. Since $X=X^*+N$,
\[
\mu\le\|X^*-Y+N\|_F\le\|Y-X^*\|_F+\|N\|_F,
\]
and nonnegativity of the norm gives the claimed maximum with zero.
\end{proof}

\section{Proofs of the Scale-Balancing Bounds}
\label{app:balancing}

\begin{proof}[Proof of Lemma~\ref{lemma:column_jump}]
Let $a=\|\hat x\|_2$, $b=\|\hat y\|_2$, and $d=\|\Delta x\|_2+\|\Delta y\|_2$. Because $\|x\|_2=\|y\|_2$, the reverse triangle inequality gives $|a-b|\le d$. If $ab>0$, the balancing factor is $s=\sqrt{b/a}$, and hence
\begin{align*}
&\|x^+-\hat x\|_2+\|y^+-\hat y\|_2=|s-1|a+|s^{-1}-1|b\\
&=|\sqrt a-\sqrt b|(\sqrt a+\sqrt b)=|a-b|\le d.
\end{align*}
If, say, $a=0$, the co-sparsity rule sets both output columns to zero. Then $\rho\le\|\Delta x\|_2$ and $b\le\rho+\|\Delta y\|_2\le d$; the case $b=0$ is symmetric.
\end{proof}

\begin{proof}[Proof of Lemma~\ref{lemma:matrix_jump}]
For column $i$, denote the two balancing displacements by $j_{u,i},j_{v,i}$ and the two preceding update lengths by $d_{u,i},d_{v,i}$. Lemma~\ref{lemma:column_jump} gives $j_{u,i}+j_{v,i}\le d_{u,i}+d_{v,i}$. Therefore
\[
j_{u,i}^2+j_{v,i}^2
\le(j_{u,i}+j_{v,i})^2
\le2(d_{u,i}^2+d_{v,i}^2),
\]
and summing over $i$ proves \eqref{eq:matrix_jump_sq}. For the nonsquared bound, collect the four sequences into nonnegative vectors $j_u,j_v,d_u,d_v$. Monotonicity of the Euclidean norm, followed by Minkowski's inequality, gives
\begin{align*}
\|j_u\|_2+\|j_v\|_2
&\le\sqrt2\|j_u+j_v\|_2\le\sqrt2\|d_u+d_v\|_2\\
&\le\sqrt2\|d_u\|_2+\sqrt2\|d_v\|_2,
\end{align*}
which is \eqref{eq:matrix_jump_l1}.
\end{proof}

\section{Proof of the P-Stationarity Characterization}
\label{app:pstationarity}

\begin{proof}[Proof of Theorem~\ref{thm:p_stationarity}]
Write $G_i=\nabla_UF_\mu(U,V)_i$, $Z_i=U_i-G_i/\gamma_{1,k}$, and $t=\sqrt{2\lambda/\gamma_{1,k}}$. The proximal problem separates by columns, and its $i$th solution set is obtained by comparing $0$ with the nonnegative projection $(Z_i)_+$:
\[
\mathcal P(Z_i)=
\begin{cases}
\{0\},&\|(Z_i)_+\|_2<t,\\
\{0,(Z_i)_+\},&\|(Z_i)_+\|_2=t,\\
\{(Z_i)_+\},&\|(Z_i)_+\|_2>t.
\end{cases}
\]
If the fixed point has $U_i=0$, then $Z_i=-G_i/\gamma_{1,k}$ and $0\in\mathcal P(Z_i)$ is equivalent to $\|[-G_i]_+\|_2\le\sqrt{2\lambda\gamma_{1,k}}$, which is (a). If $U_i\ne0$, then $U_i=(Z_i)_+$ and $\|U_i\|_2\ge t$. Elementwise, the identity
\[
U_i=(U_i-G_i/\gamma_{1,k})_+
\]
is equivalent to $U_i\ge0$, $G_i\ge0$, and $U_i\odot G_i=0$, giving (b). Conversely, (a) places $0$ in $\mathcal P(Z_i)$, while (b) gives $U_i=(Z_i)_+\in\mathcal P(Z_i)$; hence either condition implies the fixed-point relation.
\end{proof}

\section{Additional Proofs for the Structural Results}
\label{app:additional_structural_proofs}

\subsection{Proof of Theorem~\ref{thm:nmf_norm_bound}}
\label{app:proof_nmf_norm_bound}

\begin{proof}
Write $X=\sum_{k=1}^r X_k$ with $X_k=u_kv_k^\top$, where $u_k$ and $v_k$ are the $k$th columns of $U$ and $V$. The first two inequalities follow from the variational definitions of the nuclear and nonnegative nuclear norms. Balance of each pair gives
\[
\frac12(\|U\|_F^2+\|V\|_F^2)
=\sum_{k=1}^r\|u_k\|_2\|v_k\|_2
=\sum_{k=1}^r\|X_k\|_F.
\]
Since $X_k\ge0$, all cross inner products are nonnegative; hence
\[
\sum_{k=1}^r\|X_k\|_F^2
\le\left\|\sum_{k=1}^rX_k\right\|_F^2
=\|X\|_F^2.
\]
Cauchy--Schwarz now yields
$\sum_k\|X_k\|_F\le\sqrt r(\sum_k\|X_k\|_F^2)^{1/2}
\le\sqrt r\|X\|_F$.
\end{proof}

\subsection{Proof of Proposition~\ref{prop:exact_recovery}}
\label{app:proof_exact_recovery}

\begin{proof}
The case $r^*=0$ is immediate, so assume $r^*\ge1$ and choose
\begin{equation}
0<\mu<\frac{\lambda}{\sqrt{r^*}\|X^*\|_F}.
\label{eq:sup_mu_choice}
\end{equation}
This is permitted by the phrase ``sufficiently small'' in the proposition.
Let $q$ be the common number of active columns of $\overline U$ and $\overline V$; the equality of their supports follows from Proposition~\ref{prop:crit}. Dropping all nonnegative terms except the column penalty from the assumed energy bound gives
\[
2\lambda q
\le \tfrac12\|N\|_F^2
   +\mu\sqrt{r^*}\|X^*\|_F+2\lambda r^*
<2\lambda(r^*+1),
\]
where $\|N\|_F^2/2\le\lambda$ and \eqref{eq:sup_mu_choice} were used. Since $q$ is integral, $q\le r^*$, and therefore
\[
\bar r:=\rank_+(\overline U\overline V^\top)\le q\le r^*.
\]

It remains to exclude $\bar r<r^*$. Put
$s=\sigma_{r^*}(X)$. Weyl's inequality and the noise assumption imply
\begin{equation}
s\ge\sigma_{r^*}(X^*)-\|N\|_F>0,
\qquad
\|N\|_F\le s/4.
\label{eq:sup_weyl}
\end{equation}
For $\delta=r^*-\bar r\ge1$, relaxation from nonnegative rank to algebraic rank and the Eckart--Young--Mirsky theorem give
\begin{align}
\Psi_{\lambda,\mu}(\overline U,\overline V)
&\ge \min_{\rank(Z)\le\bar r}\frac12\|X-Z\|_F^2
     +2\lambda\bar r \notag\\
&=\frac12\sum_{i>\bar r}\sigma_i^2(X)+2\lambda\bar r
 \ge\frac{\delta}{2}s^2+2\lambda\bar r.
\label{eq:sup_rank_lower}
\end{align}
Moreover, the upper bound on $\lambda$ and \eqref{eq:sup_weyl} imply
$\lambda\le s^2/32$. Subtracting the assumed energy threshold from the last expression in \eqref{eq:sup_rank_lower}, and again using \eqref{eq:sup_mu_choice}, yields
\begin{align*}
&\delta\left(\frac{s^2}{2}-2\lambda\right)
  -\frac12\|N\|_F^2
  -\mu\sqrt{r^*}\|X^*\|_F\\
&\quad>\frac{7\delta}{16}s^2-\frac1{32}s^2-\lambda
\ge\frac{3}{8}s^2>0.
\end{align*}
Thus $\Psi_{\lambda,\mu}(\overline U,\overline V)$ would exceed the prescribed energy bound, a contradiction. Hence $\bar r=r^*$.
\end{proof}

\subsection{Proof of Proposition~\ref{prop:local-minimizer}}
\label{app:proof_local_minimizer}

\begin{proof}
Suppose first that $(\overline U,\overline V)$ is a local minimizer of $\widetilde F_\mu$. A sufficiently small neighborhood cannot turn any nonzero column of $\overline U$ or $\overline V$ into a zero column. Hence, throughout that neighborhood,
\[
\|U\|_{2,0}\ge\|\overline U\|_{2,0},
\qquad
\|V\|_{2,0}\ge\|\overline V\|_{2,0}.
\]
Adding these two locally nondecreasing cardinality terms to the local-minimum inequality for $\widetilde F_\mu$ proves local minimality for $\Psi_{\lambda,\mu}$. If the former inequality has quadratic growth with constant $\alpha>0$, the same $\alpha$ proves strong local minimality.

Conversely, let $(\overline U,\overline V)$ be a nonzero local minimizer of $\Psi_{\lambda,\mu}$. Proposition~\ref{prop:crit} gives a common active set $J$. For $(A,B)$ sufficiently close to $(\overline U_J,\overline V_J)$, every column remains nonzero. Zero-padding $(A,B)$ outside $J$ therefore produces $(U,V)$ near $(\overline U,\overline V)$ and
\[
\Psi_{\lambda,\mu}(U,V)=\widetilde F_\mu(A,B)+2\lambda|J|.
\]
Subtracting the same constant at $(\overline U,\overline V)$ transfers the local-minimum inequality to $(\overline U_J,\overline V_J)$. The zero-padding is an isometry, so the quadratic-growth inequality transfers unchanged in the strong case.
\end{proof}

\section{Detailed Proofs for iPALM and Scale-Balanced iPALM}
\label{app:detailed_ipalm_proofs}

\subsection{Proof of Proposition~\ref{prop:KL_properties-iPALM}}
\label{app:proof_kl_ipalm}

\begin{proof}
For compactness, define
\[
d_U^k=U^k-U^{k-1},\quad e_U^{k+1}=U^{k+1}-U^k,
\]
and analogously $d_V^k,e_V^{k+1}$. Optimality of the two proximal subproblems, followed by the block descent lemma, gives
\begin{multline}
\Psi(U^{k+1},V^{k+1})-\Psi(U^k,V^k)\\
\le \frac{\gamma_{1,k}}2\|\beta_kd_U^k\|_F^2
   -\frac{\alpha_{1,k}}2\|e_U^{k+1}-\beta_kd_U^k\|_F^2\\
\quad+\frac{\gamma_{2,k}}2\|\beta_kd_V^k\|_F^2
   -\frac{\alpha_{2,k}}2\|e_V^{k+1}-\beta_kd_V^k\|_F^2.
\label{eq:sup_ipalm_basic}
\end{multline}
Indeed, set $G_U^k=\nabla_UF_\mu(\widetilde U^k,V^k)$. For the $U$ block, this follows by adding
\begin{multline*}
\langle G_U^k,U^{k+1}-U^k\rangle
+\frac{\gamma_{1,k}}2\|U^{k+1}-\widetilde U^k\|_F^2\\
-\frac{\gamma_{1,k}}2\|U^k-\widetilde U^k\|_F^2
+g_1(U^{k+1})-g_1(U^k)\le0
\end{multline*}
to the standard smooth descent bound; the $V$ block is identical.

Expanding \eqref{eq:sup_ipalm_basic} and adding the two quadratic terms in $\Xi$ yields
\begin{align}
\Delta\Xi^k
&\le \frac{\tau_{V^k}\beta_k^2-\rho_1\alpha_2}{2}\|d_U^k\|_F^2 \notag\\
&\quad-\frac{\alpha_{1,k}-\rho_1\alpha_2}{2}\|e_U^{k+1}\|_F^2 \notag\\
&\quad+\alpha_{1,k}\beta_k\langle e_U^{k+1},d_U^k\rangle \notag\\
&\quad+\frac{\tau_{U^{k+1}}\beta_k^2-\rho_2\alpha_2}{2}\|d_V^k\|_F^2 \notag\\
&\quad-\frac{\alpha_{2,k}-\rho_2\alpha_2}{2}\|e_V^{k+1}\|_F^2 \notag\\
&\quad+\alpha_{2,k}\beta_k\langle e_V^{k+1},d_V^k\rangle.
\label{eq:sup_delta_xi}
\end{align}
Apply Young's inequality to the first cross term with
$t_{1,k}=(\alpha_{1,k}-\rho_1\alpha_2)/\alpha_{1,k}^2$, and analogously with
$t_{2,k}=(\alpha_{2,k}-\rho_2\alpha_2)/\alpha_{2,k}^2$. The $e$-terms cancel, and the remaining coefficients are exactly $-\nu_{1,k}$ and $-\nu_{2,k}$. The stated upper bound on $\beta_k$, together with
$\alpha_{i,k}\in[\alpha_1,\alpha_2]$ and the definition of $\tau$, keeps both coefficients uniformly positive. This proves (i).

Summing the decrease inequality shows that
$\sum_k(\|d_U^k\|_F^2+\|d_V^k\|_F^2)<\infty$. Coercivity of $\Psi$ (because $\mu>0$) and monotonicity of $\Xi$ imply boundedness, so the cluster set $\Upsilon$ is nonempty and compact. The successive differences vanish. Along any convergent subsequence, proximal optimality gives the matching upper-semicontinuity inequality for $g_1$ and $g_2$; combined with their lower semicontinuity, this gives convergence of both nonsmooth terms. Consequently $\Xi$ is constant on $\Upsilon$ and equals the limit of $\Xi^k$, proving (ii).

For (iii), the $U$-optimality condition is
\[
0\in\nabla_UF_\mu(\widetilde U^k,V^k)
  +\gamma_{1,k}(U^{k+1}-\widetilde U^k)+\partial g_1(U^{k+1}).
\]
It yields the following first component of a subgradient of the augmented potential:
\begin{multline*}
\Gamma_U^{k+1}
=\nabla_UF_\mu(U^{k+1},V^{k+1})
 -\nabla_UF_\mu(\widetilde U^k,V^k)\\
 -\gamma_{1,k}(U^{k+1}-\widetilde U^k)
 +\rho_1\alpha_2(U^{k+1}-U^k).
\end{multline*}
Block Lipschitz continuity on the bounded trajectory, together with
$\widetilde U^k-U^k=\beta_kd_U^k$, bounds this component by a fixed linear combination of
$\|e_U^{k+1}\|_F,\|d_U^k\|_F$, and $\|e_V^{k+1}\|_F$. The symmetric construction for $V$, plus the last two components
$\rho_1\alpha_2(U^k-U^{k+1})$ and
$\rho_2\alpha_2(V^k-V^{k+1})$, gives the stated inequality with the uniform constants $c_1$ and $c_2$.
\end{proof}

\subsection{Proof of Theorem~\ref{thm:unified_global_convergence}}
\label{app:proof_unified_convergence}

\begin{proof}
The decrease condition and lower boundedness of $\psi$ give
$\sum_k(\mathcal R^k)^2<\infty$, hence $\mathcal R^k\to0$. The step bound then implies $\Delta Z^k\to0$. Let $\psi^*=\lim_k\psi(X^k)$; in the augmented case every cluster point has the form $(Z^*,Z^*)$.

If $\psi(X^{k_0})=\psi^*$ for some $k_0$, sufficient decrease gives
$\mathcal R^k=0$ for all $k\ge k_0$, and the claim is immediate. Otherwise, the uniformized KL inequality on the compact cluster set provides a concave desingularizing function $\phi$ such that, for all sufficiently large $k$,
\[
\phi'(\psi(X^{k+1})-\psi^*)\|W^{k+1}\|_F\ge1.
\]
Set
$D_k=\phi(\psi(X^k)-\psi^*)-
\phi(\psi(X^{k+1})-\psi^*)$. Concavity, sufficient decrease, and the subgradient bound imply
\[
D_k\ge
\frac{a(\mathcal R^k)^2}{b(\mathcal R^k+\mathcal R^{k+1})}.
\]
Using $\sqrt{xy}\le(x+y)/2$ gives
\begin{align*}
\mathcal R^k
&\le\frac14(\mathcal R^k+\mathcal R^{k+1})+\frac baD_k,\\
3\mathcal R^k-\mathcal R^{k+1}
&\le\frac{4b}{a}D_k.
\end{align*}
Summing from $K$ to $N$ and telescoping $D_k$ yields
\begin{align*}
&2\sum_{k=K}^N\mathcal R^k+\mathcal R^K-\mathcal R^{N+1}\\
&\quad\le\frac{4b}{a}\big[\phi(\psi(X^K)-\psi^*)\\
&\hspace{7em}-\phi(\psi(X^{N+1})-\psi^*)\big].
\end{align*}
Letting $N\to\infty$ proves $\sum_k\mathcal R^k<\infty$. Therefore
\[
\sum_k\|\Delta Z^k\|_F\le c\sum_k\mathcal R^k<\infty,
\]
so $Z^k$ is Cauchy and converges to a single limit $Z^*$. Finally, the subgradient bound gives $W^{k+1}\to0$; closedness of the limiting subdifferential along the convergent objective values yields $0\in\partial\psi(Z^*)$ (or at $(Z^*,Z^*)$ for the augmented potential).
\end{proof}

\subsection{Proof of Proposition~\ref{prop:iPALM_scale}}
\label{app:proof_scale_ipalm}

\begin{proof}
Let $\hat U^{k+1},\hat V^{k+1}$ denote the outputs before scale balancing. Repeating the derivation of \eqref{eq:sup_ipalm_basic} for these intermediate variables and using the scaling optimality in Remark~\ref{remark2-scale}(a) gives the same descent bound for $\Psi(U^{k+1},V^{k+1})$. Lemma~\ref{lemma:matrix_jump}, followed by
$\|a+b\|_F^2\le2\|a\|_F^2+2\|b\|_F^2$, bounds the new potential terms and yields
\begin{multline}
\Delta\Xi^k
\le\frac{\rho\alpha_2}{2}
 \big(3\|\hat U^{k+1}-U^k\|_F^2-\|U^k-U^{k-1}\|_F^2\big)\\
+\frac{\rho\alpha_2}{2}
 \big(3\|\hat V^{k+1}-V^k\|_F^2-\|V^k-V^{k-1}\|_F^2\big)\\
+\frac{\gamma_{1,k}}2\|U^k-\widetilde U^k\|_F^2
 -\frac{\alpha_{1,k}}2\|\hat U^{k+1}-\widetilde U^k\|_F^2\\
+\frac{\gamma_{2,k}}2\|V^k-\widetilde V^k\|_F^2
 -\frac{\alpha_{2,k}}2\|\hat V^{k+1}-\widetilde V^k\|_F^2.
\label{eq:sup_scale_delta}
\end{multline}
Expanding the extrapolations and applying Young's inequality as in the proof of Proposition~\ref{prop:KL_properties-iPALM}, now with
$t_{i,k}=(\alpha_{i,k}-3\rho\alpha_2)/\alpha_{i,k}^2$, gives exactly the coefficients $\nu_{1,k}$ and $\nu_{2,k}$ in the proposition. The stated bound on $\beta_k$ is conservative for $3\rho$ and keeps both coefficients uniformly positive, proving (i).

The decrease inequality and coercivity give boundedness and
$U^{k+1}-U^k\to0$, $V^{k+1}-V^k\to0$. In \eqref{eq:sup_scale_delta}, the coefficients of
$\|\hat U^{k+1}-U^k\|_F^2$ and
$\|\hat V^{k+1}-V^k\|_F^2$ are bounded above by a negative constant because
$3\rho\alpha_2<\alpha_1/2$. Hence both intermediate differences vanish; Lemma~\ref{lemma:matrix_jump} also makes the balancing displacement vanish. The same lower-semicontinuity and proximal-optimality argument used above shows that $\Xi$ is constant on the compact cluster set. This proves (ii).

For (iii), the proximal optimality condition at $\hat U^{k+1}$ constructs a valid subgradient at $U^{k+1}$ because positive column scaling preserves the zero pattern, the normal cone, and the local subdifferential of the column $\ell_{2,0}$ term. One may take
\begin{multline*}
\Gamma_U^{k+1}
=\nabla_UF_\mu(U^{k+1},V^{k+1})
 -\nabla_UF_\mu(\widetilde U^k,V^k)\\
 -\gamma_{1,k}(\hat U^{k+1}-\widetilde U^k)
 +\rho\alpha_2(U^{k+1}-U^k),
\end{multline*}
and construct the $V$ component symmetrically. Block Lipschitz bounds on the compact trajectory, the extrapolation identities, and Lemma~\ref{lemma:matrix_jump} give the asserted residual bound with the stated uniform $c_1,c_2$.
\end{proof}

\section{Detailed Proofs for the Proximal Active-Set Method}
\label{app:detailed_pasm_proofs}

\subsection{HALS descent and subgradient bounds}
\label{app:proof_hals_bounds}

\begin{proof}[Proof of Lemma~\ref{lem:descent}]
For the $i$th column, $\nabla^2_{u_i}h=(\|v_i\|_2^2+\mu)I$. Exact minimization over $u_i\ge0$ therefore gives
\begin{align*}
&\mathcal L(u_{<i}^{k+1},u_{\ge i}^{k})
-\mathcal L(u_{\le i}^{k+1},u_{>i}^{k})\\
&\quad\ge\frac{\|v_i\|_2^2+\mu}{2}
\|u_i^{k+1}-u_i^k\|_2^2.
\end{align*}
Summing over the cyclic sweep proves the result with
$c_1=\frac12\min_i(\|v_i\|_2^2+\mu)\ge\mu/2$.
\end{proof}

\begin{proof}[Proof of Lemma~\ref{lem:subgradient}]
Column optimality gives
$0\in\nabla_{u_i}h(u_{\le i}^{k+1},u_{>i}^k)+
\partial\delta_{\mathbb R_+}(u_i^{k+1})$. Set $B=V^\top V$. A subgradient at the end of the sweep is then
\begin{align*}
W_{u_i}^{k+1}
&=\nabla_{u_i}h(U^{k+1})
 -\nabla_{u_i}h(u_{\le i}^{k+1},u_{>i}^{k})\\
&=\sum_{j>i}(u_j^{k+1}-u_j^k)B_{ji}.
\end{align*}
Equivalently, $W^{k+1}=(U^{k+1}-U^k)M_L$, where $M_L$ is the strictly lower triangular part of $B$; the claimed norm bound follows immediately.
\end{proof}

\subsection{PASM sufficient decrease and subgradient bound}
\label{app:proof_pasm_bounds}

\begin{proof}[Proof of Lemma~\ref{lem:sufficient_decrease}]
The hard-thresholding step is the noninertial proximal update, so
\[
\Psi(\widetilde U^{k+1},V^k)
\le\Psi(U^k,V^k)-\frac{\alpha_1}{2}
\|\widetilde U^{k+1}-U^k\|_F^2.
\]
The HALS sweep stays on the same active support, where the $\ell_{2,0}$ term is constant. Lemma~\ref{lem:descent} therefore gives
\[
\Psi(\hat U^{k+1},V^k)
\le\Psi(\widetilde U^{k+1},V^k)-\frac\mu2
\|\hat U^{k+1}-\widetilde U^{k+1}\|_F^2.
\]
The symmetric two inequalities hold for the $V$ update. Adding all four and using the fact that scale balancing cannot increase $\Psi$ proves the result with $\rho=\min(\alpha_1,\mu)$.
\end{proof}

\begin{proof}[Proof of Lemma~\ref{lem:subgradient_bound}]
Let $\mathcal T=\mathcal T_U^k$. On inactive columns, the limiting subdifferential of the column $\ell_{2,0}$ term is the entire ambient space, so the $U$ component of a subgradient of $\Psi$ can be selected as zero. On $\mathcal T$, Lemma~\ref{lem:subgradient} supplies a HALS subgradient
$\widetilde W_{U,\mathcal T}^{k+1}$ satisfying
\begin{align*}
\|\widetilde W_{U,\mathcal T}^{k+1}\|_F
&\le C_V\|\hat U^{k+1}-\widetilde U^{k+1}\|_F,\\
C_V&=\sup_k\|M_{L,V^k}\|_2.
\end{align*}
Correcting this element for the subsequent $V$ update and scale balancing gives
\begin{align*}
(W_U^{k+1})_{\mathcal T}
&=\widetilde W_{U,\mathcal T}^{k+1}\\
&\quad+\nabla_{U_{\mathcal T}}F_\mu(U^{k+1},V^{k+1})\\
&\quad-\nabla_{U_{\mathcal T}}F_\mu(\hat U^{k+1},V^k).
\end{align*}
Consequently, on the bounded trajectory,
\begin{multline*}
\|W_U^{k+1}\|_F
\le C_V\|\hat U^{k+1}-\widetilde U^{k+1}\|_F\\
+c_f\|V^{k+1}-\hat V^{k+1}\|_F
+\mu\|U^{k+1}-\hat U^{k+1}\|_F\\
 +(c_f+\tau)\|\hat V^{k+1}-V^k\|_F.
\end{multline*}
The symmetric construction gives the $V$ bound with
$C_U=\sup_k\|M_{L,\hat U^{k+1}}\|_2$. Finally, Lemma~\ref{lemma:matrix_jump} controls both balancing displacements, while
\[
\|\hat U^{k+1}-U^k\|_F
\le\|\hat U^{k+1}-\widetilde U^{k+1}\|_F
   +\|\widetilde U^{k+1}-U^k\|_F
\]
and the analogous $V$ inequality reduce all terms to the four residuals in the statement. Collecting coefficients gives
$C=C_V+C_U+\sqrt2\mu+\tau+(\sqrt2+1)c_f$.
\end{proof}

\bibliographystyle{IEEEtran}
\bibliography{references}

\vfill

\end{document}